\documentclass[seceqn,secthm]{elsart}

\usepackage{pifont}
\usepackage{amssymb,amsmath,amsfonts,amsbsy,calrsfs}
\usepackage{graphicx,natbib}
\usepackage{eurosym}

\newenvironment{proof}{\paragraph*{Proof}}{\mbox{}\hfill$\Box$}
\newtheorem{definition}[thm]{Definition}
\newtheorem{theorem}[thm]{Theorem}
\newtheorem{proposition}[thm]{Proposition}
\newtheorem{example}[thm]{Example}

\renewcommand{\leq}{\leqslant}
\renewcommand{\geq}{\geqslant}

\newcommand{\EE}{\mathbb{E}}

\newcommand{\RR}{\mathbb{R}}

\newcommand{\RV}{{\cal R}}

\newcommand{\bF}{\overline{F}}

\newcommand{\bG}{\overline{G}}

\newcommand{\bH}{\overline{H}}

\newcommand{\Y}{\check{Y}}
\newcommand{\Ynk}{\Y_{n-k:n}}

\DeclareMathOperator{\cov}{cov}
\DeclareMathOperator{\E}{E}
\DeclareMathOperator{\var}{var}

\newcommand{\eps}{\varepsilon}

\newcommand{\dto}{\rightsquigarrow}
\newcommand{\eqd}{\stackrel{d}{=}}
\newcommand{\rmd}{\, \mathrm{d}}

\begin{document}

\begin{frontmatter}

\title{Second-Order Refined Peaks-Over-Threshold Modelling for Heavy-Tailed Distributions}

\author[KUL]{Jan Beirlant\thanksref{IAP}}
\ead{jan.beirlant@wis.kuleuven.be}
\address[KUL]{Department of Mathematics and Leuven Statistics Research Centre \\
Katho\-lie\-ke Uni\-ver\-si\-teit Leuven, Celestijnenlaan 200b, B-3001 Heverlee, Belgium}

\author[JRC]{Elisabeth Joossens}
\ead{elisabeth.joossens@jrc.it}
\address[JRC]{Joint Research Centre, European Commission \\ 
Via Fermi 2749, 21027 Ispra (VA), Italy}

\author[UCL]{Johan Segers\corauthref{cor}\thanksref{IAP}}
\ead{johan.segers@uclouvain.be}
\thanks[IAP]{Supported by IAP research network grant no.\ P6/03 of the Belgian government (Belgian Science Policy).}
\corauth[cor]{Corresponding author.}
\address[UCL]{Institut de statistique, Universit\'e catholique de Louvain \\
Voie du Roman Pays 20, B-1348 Louvain-la-Neuve, Belgium}

\begin{abstract}
Modelling excesses over a high threshold using the Pareto or generalized Pareto distribution (PD/GPD) is the most popular approach in extreme value statistics. This method typically requires high thresholds in order for the (G)PD to fit well and in such a case applies only to a small upper fraction of the data. The extension of the (G)PD  proposed in this paper is able to describe the excess distribution for lower thresholds in case of heavy tailed distributions. This yields a statistical model that can be fitted to a larger portion of the data. Moreover, estimates of tail parameters display stability for a larger range of thresholds. Our findings are supported by asymptotic results, simulations and a case study.
\end{abstract}

\end{frontmatter}

\section{Introduction}

It is well known that a distribution is in the max-domain of attraction of an extreme value distribution if and only if the distribution of excesses over high thresholds is asymptotically generalized Pareto (GP) \citep{BalkemadeHaan74, Pickands75}. This result gave rise to the peaks-over-threshold methodology introduced in \citet{DavisonSmith90}; see also \citet{Coles01}. The method consists of two components: modelling of clusters of high-threshold exceedances with a Poisson process and modelling of excesses associated to the cluster peaks with a GPD. In practice, a way to verify the validity of the model is to check whether the estimates of the GP shape parameter are stable when the model is fitted to excesses over a range of thresholds. The question then arises how to proceed if this threshold stability is not visible for a given data set. From a theoretical point of view, absence of the stability property can be explained by a slow rate of convergence in the Pickands--Balkema--de Haan theorem. In case of heavy-tailed distributions, the same issue arises when fitting a Pareto distribution (PD) to the relative excesses over high, positive thresholds.

A possible solution is to build a more flexible model capable of capturing the deviation between the true excess distribution and the asymptotic model. For heavy-tailed distributions, this deviation can be parametrized using a power series expansion of the tail function \citep{Hall82}, or more generally via second-order regular variation \citep{GdH87, BGT}.

The aim of this paper is to propose such an extension, called the extended Pareto or extended generalized Pareto distribution (EPD/EGPD). A key distinction with other approaches is that although in previous papers the second-order approximation is used for adjusting the inference of the tail index, inference on the tail itself is still based on the GPD; in contrast, in our approach the EP(G)D is fitted directly to the high-threshold excesses. Indeed, as we will show later, even if the (G)PD parameters are estimated in an unbiased way, tail probability estimators may still exhibit asymptotic bias if based upon the (G)PD approximation. 

The main advantages of the new model are a reduction of the bias of estimators of tail parameters and a good fit to excesses over a larger range of thresholds. In an actuarial context, the relevance of using more elaborate models has already been discussed for instance in \citet{FHR02} and \citet{CoorayAnanda05}.

In case of heavy-tailed distributions, it is more convenient to work with relative excesses $X/u$ rather than absolute excesses $X-u$. Under the domain of attraction condition the limit distribution of $X/u$ given $X > u$ for $u \to \infty$ is the PD. The EPD and EGPD presented here are related through the same affine transformation that links these relative and absolute excesses. Building on the theory of generalized regular variation of second order in \citet{dHS96}, it is also possible to construct an extension of the GPD with comparable merits applicable to distributions in all max-domains of attraction. However, parameter estimation in this more general setting is numerically quite involved \citep{BJS02}: the model contains one additional parameter and the upper endpoint of the distribution depends in a complicated way on the parameters, which complicates both theory and computations.

Bias-reduction methods have already been proposed in, amongst others, \citet{FeuervergerHall99}, \citet{GMN00}, \citet{BDGM99}, \citet{DBGS02}, \citet{GM02}, and \citet{GM04}. These methods focus on the distribution of log-spacings of high order statistics. Moreover, \textit{ad hoc} construction methods for asymptotically unbiased estimators of the extreme value index were introduced in \citet{Peng98}, \citet{Drees96} and \citet{Segers05}. In contrast, next to providing bias-reduced tail index estimators, our model can be fitted directly to the excesses over a high threshold. The fitted model can then be used to estimate any tail-related risk measure, such as tail probabilities, tail quantiles (or value-at-risk), etc. 

In the same spirit as in this paper, a mixture model with two Pareto components was proposed in \citet{PengQi04}. The advantage of our model is that it also incorporates the popular GPD. From our experience, this connection can assist in judging the quality of the GPD fit; see for instance the case study in Example~\ref{Ex:secura}.

The paper is structured as follows. The next section provides the definition of the E(G)PD, which is shown to yield a more accurate approximation to the distribution of absolute and relative excesses for a wide class of heavy-tailed distributions. Estimators of the EPD parameters are derived in Section~\ref{S:par} using the linearized score equations, and their asymptotic normality is formally stated. In Section~\ref{S:compar}, we compare the asymptotic distribution and the finite-sample behavior of the estimators of the extreme value index following from PD, GPD and EPD modelling. To illustrate how to apply the methodology to the estimation of general tail-related risk measures, we elaborate in Section~\ref{S:prob} on tail probability estimation with theoretical results and a practical case. The appendices, finally, contain the statement and proof of an auxiliary result on a certain tail empirical process followed by the proofs of the main theorems.

\section{The Extended (Generalized) Pareto Distribution}
\label{S:EGPD}

\begin{definition}
\label{D:EGPD}
The \emph{Extended Pareto Distribution (EPD)} with parameter vector $(\gamma, \delta, \tau)$ in the range $\tau < 0 < \gamma$ and $\delta > \max(-1, 1/\tau)$ is defined by its distribution function
\[
    G_{\gamma,\delta,\tau}(y) =
    \begin{cases}
        1 - \{y(1 + \delta - \delta y^\tau)\}^{-1/\gamma}, & \text{if $y > 1$}, \\
        0, & \text{if $y \leq 1$.}
    \end{cases}
\]
The \emph{Extended Generalized Pareto Distribution (EGPD)} is defined by its distribution function
\[
    H_{\gamma,\delta,\tau}(x) = G_{\gamma,\delta,\tau}(1+x), \qquad x \in \RR.
\]
\end{definition}

The ordinary Pareto Distribution (PD) with shape parameter $\alpha
> 0$ is a member of the EPD family: take $\gamma = 1/\alpha$ and
$\delta = 0$ (arbitrary $\tau$). The Generalized Pareto Distribution
(GPD) with positive shape parameter $\gamma > 0$ and scale parameter
$\sigma > 0$ is a member of the EGPD family: take $\tau = -1$ and
$\delta = \gamma / \sigma - 1$. Finally, the distribution of the
random variable $Y$ is EPD($\gamma,\delta,\tau$) if and only if the
distribution of $Y - 1$ is EGPD($\gamma,\delta,\tau$).

We will use the E(G)PD to model tails of heavy-tailed distributions
that satisfy a certain second-order condition, to be described next.
For a distribution function $F$, write $\bF = 1-F$. Recall that a
positive, measurable function $f$ defined in some right neighborhood
of infinity is \emph{regularly varying} with index $\beta \in \RR$
if $\lim_{u \to \infty} f(ux)/f(u) = x^\beta$ for all $x \in (0,
\infty)$; notation $f \in \RV_\beta$. The following definition describes a subset of the
class of distribution functions $F$ for which $\bF \in
\RV_{-1/\gamma}$, $\gamma > 0$. Note that the latter is precisely the class of distributions in the max-domain of attraction of the Fr\'echet distribution with shape parameter $1/\gamma$.

\begin{definition}
\label{C:2nd}
Let $\gamma > 0$ and $\tau < 0$ be constants. A distribution function $F$ is said to belong to the class ${\cal F}(\gamma, \tau)$ if $x^{1/\gamma} \bF(x) \to C \in (0, \infty)$ as $x \to \infty$ and if the function $\delta$ defined via
\begin{equation}
\label{E:delta}
    \bF(x) = C x^{-1/\gamma} \{ 1 + \gamma^{-1} \delta(x) \}
\end{equation}
is eventually nonzero and of constant sign and such that $|\delta| \in \RV_\tau$.
\end{definition}

Note that $|\delta| \in \RV_\tau$ with $\tau < 0$ implies $\delta(x)
\to 0$ as $x \to \infty$. In many examples, the function $\delta$ in
Definition~\ref{C:2nd} is actually of the form $\delta(x) \sim D
x^\tau$ as $x \to \infty$ for some nonzero constant $D$, a class of
distributions which was first considered in \citet{Hall82}. See
Table~\ref{tab:parameters} for examples; for later use, we also list
$\rho = \gamma \tau$ (see Lemma~\ref{L:rho} below).

\begin{table}
\begin{center}
\begin{tabular}{@{}l l l l l}
\hline \hline \\[-2em]
distribution & distribution function & $\gamma$ & $\tau$ & $\rho = \gamma \tau$ \\
{[parameters]} \\ \hline \\[-1em]
Burr($\gamma,\rho,\beta$) &
    $1 - (1+x^{-\rho/\gamma}/\beta)^{1/\rho}$ &
    $\gamma$ & $\phantom{-}\rho / \gamma$ & $\phantom{-}\rho$ \\
    {[$\gamma>0$, $\rho<0$, $\beta>0$]} \\[1ex]
Fr\'echet($\alpha$) &
    $\exp(-x^{-\alpha})$ & $1/\alpha$ & $-\alpha$ & $-1$ \\
    {[$\alpha>0$]} \\[1ex]
GPD($\gamma,\sigma$) &
    $1 - (1 + \gamma x/\sigma)^{-1/\gamma}$ &
    $\gamma$ & $-1$ & $-\gamma$ \\
    {[$\gamma > 0$, $\sigma > 0$]} \\[1ex]
Student-\emph{t}$_{\nu}$ &
    $C(\nu) \int_{-\infty}^x (1+\frac{y^2}{\nu})^{-(\nu+1)/2} \rmd y$ &
    $1/\nu$ & $-2$ & $-2/\nu$ \\
    {[$\nu>0$]} \\
\hline \hline \\
\end{tabular}
\caption{\it Extreme value index $\gamma$ and second-order constants $\tau$ and $\rho = \gamma \tau$ for selected heavy-tailed distributions. 
\label{tab:parameters}}
\end{center}
\end{table}

Let $X$ be a random variable with distribution function $F$ and let $u > 0$ be such that $F(u) < 1$. The conditional distributions of relative and absolute excesses of $X$ over $u$ are given by
\[
    \Pr(X/u > y \mid X > u) = \frac{\bF(uy)}{\bF(u)} \quad \text{and} \quad
    \Pr(X-u > x \mid X > u) = \frac{\bF(u+x)}{\bF(u)}
\]
for $x \geq 0$ and $y \geq 1$. The next proposition shows that for $F \in {\cal F}(\gamma, \tau)$, the EPD and the EGPD improve the PD and GPD approximations to these excess distributions with an order of magnitude.

\begin{proposition}
\label{P:EGPD}
If $F \in {\cal F}(\gamma, \tau)$, then as $u \to \infty$,
\begin{align}
\label{E:P:EPD}
    \sup_{y \geq 1}
    \biggl| \frac{\bF(uy)}{\bF(u)} - \bG_{\gamma,\delta(u),\tau}(y) \biggr| &= o\{|\delta(u)|\}, \\
\label{E:P:EGPD}
    \sup_{x \geq 0}
    \biggl| \frac{\bF(u+x)}{\bF(u)} - \bH_{\gamma,\delta(u),\tau}(x/u) \biggr| &= o\{|\delta(u)|\}.
\end{align}
\end{proposition}

\begin{proof}
Equation~\eqref{E:P:EGPD} follows directly from \eqref{E:P:EPD} by writing $u + x = uy$ or $y = 1 + x/u$ and exploiting the link between the EPD and the EGPD. So let us show \eqref{E:P:EPD}. On the one hand, we have
\[
    \frac{\bF(uy)}{\bF(u)}
    = y^{-1/\gamma} \frac{1 + \gamma^{-1} \delta(uy)}{1 + \gamma^{-1} \delta(u)}
    = y^{-1/\gamma} \left( 1 - \gamma^{-1} \delta(u) \frac{1 - \frac{\delta(uy)}{\delta(u)}}{1 + \gamma^{-1} \delta(u)} \right).
\]
On the other hand, since $0 \leq 1 - y^\tau \leq 1$ for $y \geq 1$ and since $\delta(u) \to 0$,
\begin{multline*}
	[y \{ 1 + \delta(u) - \delta(u) y^\tau \}]^{-1/\gamma} \\
  = y^{-1/\gamma} \{1 - \gamma^{-1} \delta(u) (1 - y^\tau)\} + o\{|\delta(u)|\}, \qquad u \to \infty,
\end{multline*}
uniformly in $y \geq 1$. As a consequence,
\begin{multline*}
    \frac{\bF(uy)}{\bF(u)}
    - [y \{ 1 + \delta(u) - \delta(u) y^\tau \}]^{-1/\gamma} \\
    = - \gamma^{-1} y^{-1/\gamma} \delta(u)
    \left( \frac{1 - \frac{\delta(uy)}{\delta(u)}}{1 + \gamma^{-1} \delta(u)} - (1 - y^\tau) \right)
    + o\{|\delta(u)|\}, \qquad u \to \infty,
\end{multline*}
uniformly in $y \geq 1$. The asymptotic relation~\eqref{E:P:EPD} now follows from the uniform convergence theorem for regularly varying functions with negative index \citep[Theorem~1.5.2]{BGT}.
\end{proof}

If in \eqref{E:P:EPD} we would replace the EPD tail function $\bG_{\gamma,\delta(u),\tau}(y)$ by the PD tail function $y^{-1/\gamma}$, the rate of convergence would be $O\{|\delta(u)|\}$ only. Similarly, if in \eqref{E:P:EGPD} we would replace the EGPD tail function $\bH_{\gamma,\delta(u),\tau}(x/u)$ by the GPD tail function $(1 + \gamma x / \sigma)^{-1/\gamma}$ for some $\sigma = \sigma(u)$, then, provided $\tau \neq -1$, the rate of convergence would again be $O\{|\delta(u)|\}$ only. If $\tau = -1$, the EGPD is just a reparametrization of the GPD, so that in that case, the GPD approximation is already of the order $o\{|\delta(u)|\}$.

It will be useful to rephrase our second-order assumption on $F$ in terms of the tail quantile function $U$ defined by
\begin{equation}
\label{E:U}
	U(y) = Q(1 - 1/y) \quad \text{ with } \quad Q(p) = \inf \{ x \in \mathbb{R} : F(x) \geq p \}, 
\end{equation}
where $y \in (1, \infty)$ and $p \in (0, 1)$. Note that $U$ is a (generalized) inverse of $1/\bF$.

\begin{lem}
\label{L:rho}
If $F \in {\cal F}(\gamma, \tau)$ with $\lim_{x \to \infty} x^{1/\gamma} \bF(x) = C \in (0, \infty)$, then $\lim_{y \to \infty} y^{-\gamma} U(y) = C^\gamma$, and the function $a$ defined implicitly by
\begin{equation}
\label{E:a}
	U(y) = C^\gamma y^\gamma \{1 + a(y)\}
\end{equation}
satisfies $a(y) = \delta(U(y)) \{1 + o(1)\} = \delta(C^\gamma y^\gamma) \{1 + o(1)\}$ as $y \to \infty$, with $\delta$ as in \eqref{E:delta}.
\end{lem}

In particular, $a$ is eventually nonzero and of constant sign and $|a| \in \RV_\rho$ with $\rho = \gamma \tau < 0$. In addition, even if $F$ is not continuous, then still $y \bF(U(y)) = 1 + o\{|a(y)|\}$ as $y \to \infty$.

\section{Parameter Estimation}
\label{S:par}

Our aim is to make inference on the distribution function $F$ on the
region to the right of some high, positive threshold $u$. To this
end, we assume $F \in {\cal F}(\gamma, \tau)$ and rewrite \eqref{E:P:EPD} as follows: as $u \to \infty$ and uniformly
in $y \geq 1$,
\begin{equation}
\label{E:approx}
	\bF(uy) = \bF(u) \bG_{\gamma,\delta(u),\tau}(y) + o\{ \bF(u) |\delta(u)| \}.
\end{equation}
Omitting the remainder term leads to an approximation of $\bF(x)$
for $x \geq u$ in terms of $\bF(u)$ and the EPD parameters $(\gamma,
\delta(u), \tau)$. Replacing these unknown quantities by estimates
then yields our estimate for $\bF(x)$.

The purpose of this section is to construct estimators of the E(G)PD
parameters $(\gamma, \delta(u), \tau)$. As usual in extreme value statistics, the
threshold exceedance probability $\bF(u)$ will be estimated
nonparametrically. Although the arguments leading to the estimators
will be of a heuristic nature only, the asymptotic behaviour of the
estimators will be stated and proved rigorously.

Let $X_1, \ldots, X_n$ be a random sample from $F$. In view of \eqref{E:P:EPD}, the estimates of the EPD parameters will be based on the relative excesses $X_i / u$ over $u$, for those $i \in \{1, \ldots, n\}$ such that $X_i > u$. In an extreme value asymptotic setting, the threshold $u$ needs to tend to infinity to make the approximation valid; at the same time, in a statistical context, the number of excesses over $u$ must be sufficiently large to make inference feasible. Denoting the order statistics by $X_{1:n} \leq \cdots \leq X_{n:n}$, we can ensure both criteria to be met by choosing a data-adaptive threshold $u = u_n = X_{n-k:n}$ where $k = k_n \in \{1, \ldots, n-1\}$ is an intermediate sequence of integers, that is, $k \to \infty$ and $k/n \to 0$ as $n \to \infty$. For convenience, assume $F(0) = 0$, so that all $X_i$ are positive with probability one.

Recall the tail quantile function $U$ in \eqref{E:U} and the auxiliary function $a$ in Lemma~\ref{L:rho}. In addition to $k$ being an intermediate sequence, we will assume that
\begin{equation}
\label{E:k}
  \sqrt{k} a(n/k) \to \lambda \in \mathbb{R}, \qquad n \to \infty.
\end{equation}
Writing $\delta_n = \delta(u_n) = \delta(X_{n-k:n})$, we will show later that \eqref{E:k} implies
\begin{equation}
\label{E:kdelta}
  \sqrt{k} \delta_n = \lambda + o_p(1), \qquad n \to \infty.
\end{equation}
Since in the definition of the EPD the term $x^{\tau}$ is multiplied by $\delta$, the previous display implies that the asymptotic distribution of tail estimators based on \eqref{E:approx} will not depend on the asymptotic distribution of the estimator of $\tau$, not even on its rate of convergence. Therefore, we will assume for the moment that $\tau$ (or $\rho$) is known. In the end, the unknown second-order parameters will be replaced by consistent estimators, a substitution which will be shown not to affect the asymptotic distributions of the other estimators. Note that under the regime $\sqrt{k} |a(n/k)| \to \infty$ as $n \to \infty$, which will not be considered in this paper, the asymptotic distribution of the estimator of the second-order parameter does play a role.

The estimators of $\gamma$ and $\delta_n$ will be found by maximizing an approximation to the EPD likelihood given the sample of $k$ relative excesses $X_{n-k+i:n}/X_{n-k:n}$, $i \in \{1, \ldots, k\}$, over the random threshold $X_{n-k:n}$. The density function of the EPD is given by
\[
    g_{\gamma,\delta,\tau}(x)
    = \frac{1}{\gamma} x^{-1/\gamma-1} \{1 + \delta(1-x^\tau)\}^{-1/\gamma-1} [1 + \delta\{1 - (1+\tau)x^\tau\}].
\]
The score functions admit the following expansions in $\delta \to 0$:
\begin{align*}
    \frac{\partial}{\partial \gamma} \log g_{\gamma,\delta,\tau}(x)
    &= - \frac{1}{\gamma} + \frac{1}{\gamma^2} \log x + \frac{\delta}{\gamma^2} (1 - x^\tau) + O(\delta^2), \\
    \frac{\partial}{\partial \delta} \log g_{\gamma,\delta,\tau}(x)
    &= \frac{1}{\gamma} \{ (1 - \gamma \tau) x^\tau - 1 \} \\
    & \qquad \mbox{} + \{ 1 - 2(1 - \gamma \tau)x^\tau
    + (1 - 2 \gamma \tau - \gamma \tau^2) x^{2\tau} \} \frac{\delta}{\gamma} + O(\delta^2).
\end{align*}
Define
\begin{align}
\label{E:Hill}
    H_{k,n} &= \frac{1}{k} \sum_{i=1}^k \log (X_{n-k+i:n} / X_{n-k:n}), \\
\label{E:Ekn}
    E_{k,n}(s) &= \frac{1}{k} \sum_{i=1}^k (X_{n-k+i:n} / X_{n-k:n})^s, \qquad s \leq 0.
\end{align}
Note that $H_{k,n}$ is the Hill estimator \citep{Hill75}. Assume for the moment that $\tau$ is known. Given the sample of excesses $X_{n-k+i:n}/X_{n-k:n}$, $i = 1, \ldots, k$, solving the linearized score equations yields the following equations for the pseudo-maximum likelihood estimators for $\gamma$ and $\delta$:
\begin{align}
\label{E:score:1}
    \hat{\gamma}_{k,n}
    &= H_{k,n} + \hat{\delta}_{k,n} \{ 1 - E_{k,n}(\tau) \}, \\
\label{E:score:2}
    (\hat{\gamma}_{k,n} \tau - 1) E_{k,n}(\tau) + 1
    &=
    \{1 - 2(1 - \hat{\gamma}_{n,k} \tau)E_{k,n}(\tau) \nonumber \\
    & \qquad \mbox{}
        + (1 - 2 \hat{\gamma}_{k,n} \tau - \hat{\gamma}_{k,n} \tau^2) E_{k,n}(2\tau)\} \hat{\delta}_{k,n}.
\end{align}
Substitute the expression for $\hat{\gamma}_{k,n}$ in \eqref{E:score:1} into the left-hand side of \eqref{E:score:2} and solve for $\hat{\delta}_{k,n}$ to get
\[
    \hat{\delta}_{k,n}
    = \frac{(H_{k,n} \tau - 1) E_{k,n}(\tau) + 1}{D_{k,n}}
    = \frac{H_{k,n} \tau - 1}{D_{k,n}} \left( E_{k,n}(\tau) - \frac{1}{1 - H_{k,n} \tau} \right),
\]
the denominator being
\begin{multline*}
    D_{k,n} = 1 - 2(1 - \hat{\gamma}_{n,k} \tau)E_{k,n}(\tau)
        + (1 - 2 \hat{\gamma}_{k,n} \tau - \hat{\gamma}_{k,n} \tau^2) E_{k,n}(2\tau) \\
    - \tau \{ 1 - E_{k,n}(\tau) \} E_{k,n}(\tau).
\end{multline*}

By \eqref{E:kdelta}, $\hat{\delta}_{k,n}$ can be expected to be of the order $O_p(k^{-1/2})$ as $n \to \infty$. This justifies the following simplifications. Since the distribution of relative excesses over a large threshold is approximately Pareto with shape parameter $1/\gamma$, for $s \leq 0$,
\[
    E_{k,n}(s) = \frac{1}{1 - \gamma s} + o_p(1),
    \qquad n \to \infty;
\]
see Theorem~\ref{T:simpler}. Hence, writing $\rho = \gamma \tau$, we have $E_{k,n}(\tau) = (1 - \rho)^{-1} + o_p(1)$ and $E_{k,n}(2\tau) = (1 - 2\rho)^{-1} + o_p(1)$ as $n \to \infty$, so that
\[
    D_{k,n} = - \frac{\rho^4}{\gamma (1 - 2 \rho) (1 - \rho)^2} + o_p(1), \qquad n \to \infty.
\]
This leads to the following simplified estimators:
\begin{align*}
    \hat{\delta}_{k,n}
    &= H_{k,n} (1 - 2 \rho) (1 - \rho)^3 \rho^{-4}
        \left( E_{k,n}(\tau) - \frac{1}{1 - H_{k,n} \tau} \right), \\
    \hat{\gamma}_{k,n}
    &= H_{k,n} - \hat{\delta}_{k,n} \frac{\rho}{1 - \rho}.
\end{align*}

Up to now we have assumed that $\rho$ is known. Let $\hat{\rho}_n$ be a weakly consistent estimator sequence of $\rho = \gamma \tau$; see for instance \citet{FAdHL03}, \citet{FAGdH03}, and \citet{PengQi04}. Replace $\tau$, which is unknown, by $\hat{\tau}_{k,n} = \hat{\rho}_n / H_{k,n}$, to finally get
\begin{align}
\label{E:deltakn}
    \hat{\delta}_{k,n} &=
    H_{k,n} (1 - 2 \hat{\rho}_n) (1 - \hat{\rho}_n)^3 \hat{\rho}_n^{-4}
    \left( E_{k,n}(\hat{\rho}_n / H_{k,n}) - \frac{1}{1 - \hat{\rho}_n} \right), \\
\label{E:gammakn}
    \hat{\gamma}_{k,n} &=
    H_{k,n} - \hat{\delta}_{k,n} \frac{\hat{\rho}_n}{1 - \hat{\rho}_n}.
\end{align}
Further, put
\begin{equation}
\label{E:Zkn}
    Z_{k,n} = \sqrt{k} \{ n \bF(X_{n-k:n}) / k - 1 \}.
\end{equation}
The joint asymptotics of $Z_{k,n}$ with $(\hat{\gamma}_{k,n}, \hat{\delta}_{k,n})$ will become relevant in Section~\ref{S:prob} when estimating tail probabilities on the basis of \eqref{E:approx} with $u = X_{n-k:n}$. Let the arrow $\dto$ denote convergence in distribution.

\begin{theorem}
\label{T:estim}
Let $F \in {\cal F}(\gamma, \tau)$ and let $X_1, \ldots, X_n$ be independent random variables with common distribution function $F$. Let $k = k_n$ be an intermediate sequence satisfying \eqref{E:k}. Recall $\delta_n = \delta(X_{n-k:n})$ and $Z_{k,n}$ in \eqref{E:Zkn}. If $\hat{\rho}_n = \rho + o_p(1)$ as $n \to \infty$, with $\rho = \gamma \tau$, then $\sqrt{k} \delta_n = \lambda + o_p(1)$ as $n \to \infty$ and
\begin{equation}
\label{E:estim}
    \Bigl( \sqrt{k} (\hat{\gamma}_{k,n} - \gamma), \sqrt{k} (\hat{\delta}_{k,n} - \delta_n), Z_{k,n} \Bigr)
    \dto
    N_3(\boldsymbol{0}, \Sigma), \qquad n \to \infty,
\end{equation}
a trivariate normal distribution with mean vector zero and covariance matrix
\begin{equation}
\label{E:Sigma}
    \Sigma =
    \left(
        \begin{array}{llc}
             \phantom{-}\gamma^2 \frac{(1-\rho)^2}{\rho^2} &  -\gamma^2 \frac{(1-2\rho)(1-\rho)}{\rho^3} & 0 \\[1ex]
             -\gamma^2 \frac{(1-2\rho)(1-\rho)}{\rho^3} &  \phantom{-}\gamma^2 \frac{(1-2\rho)(1-\rho)^2}{\rho^4} & 0 \\[1ex]
            \phantom{-}0      & \phantom{-}0      & 1
        \end{array}
    \right).
\end{equation}
\end{theorem}

An asymptotic confidence interval for $\gamma$ of nominal level $1 - \alpha$ is given by
\begin{equation}
\label{E:CI:gamma}
	\biggl[ 
		\hat{\gamma}_{k,n} \biggl( 1 + \frac{1 - \hat{\rho}_n}{\hat{\rho}_n} \frac{z_{\alpha/2}}{\sqrt{k}} \biggr), \;
		\hat{\gamma}_{k,n} \biggl( 1 - \frac{1 - \hat{\rho}_n}{\hat{\rho}_n} \frac{z_{\alpha/2}}{\sqrt{k}} \biggr)
	\biggr],
\end{equation}
with $z_{\alpha/2}$ the $1 - \alpha/2$ quantile of the standard normal distribution.

The proof of Theorem~\ref{T:estim} is given in
Appendix~\ref{A:estim}. It is based on a functional central limit
theorem for a certain tail empirical process, stated and proved in
Appendix~\ref{A:simpler}. Note that the
asymptotic distribution of $\hat{\rho}_{k,n}$ is unimportant; the
only requirement is that the estimator is consistent for $\rho$.

The fact that the limit distribution in \eqref{E:estim} is centered for any $\lambda$, is important for two reasons:
\begin{itemize}
\item[\bf 1] It makes possible the use of larger $k$ and thus of lower thresholds compared to when the mean would be proportional to $\lambda$. In this way, the model can be fitted to a larger fraction of the data, leading to a reduction of the asymptotic variances and thus of the asymptotic mean squared errors of the parameter estimates.
\item[\bf 2] Sample paths of the estimates as a function of $k$ will exhibit larger regions of stability around the true value. As a consequence, the choice of $k$ becomes easier.
\end{itemize}
These issues will be illustrated in the simulations in Section~\ref{S:compar} and in the case study in Example~\ref{Ex:secura}.

\section{Comparison of Extreme Value Index Estimators}
\label{S:compar}

Under the conditions of Theorem~\ref{T:estim}, we have
\begin{equation}
\label{E:AN:EPD}
    \sqrt{k} (\hat{\gamma}_{k,n} - \gamma)
    \dto N \left( 0, \gamma^2 \frac{(1-\rho)^2}{\rho^2} \right),
    \qquad n \to \infty.
\end{equation}
According to \citet{Drees98}, the asymptotic variance is minimal for scale-invariant, asymptotically unbiased estimators of $\gamma$ of a certain form. The limit distribution in \eqref{E:AN:EPD} corresponds with the one of the estimators in \citet{BDGM99}, \citet{FeuervergerHall99} and \citet{GM02}.

The maximum likelihood estimator for $\gamma$ arises from fitting
the GPD to the excesses $X_{n-k+i:n} - X_{n-k:n}$, $i = 1, \ldots,
k$. Its asymptotics have been studied in \citet{Smith87},
\citet{DFdH04} and \citet[Theorem~3.4.2]{dHF}. From the latter
theorem, it follows that under the conditions of our
Theorem~\ref{T:estim}, we have
\begin{equation}
\label{E:AN:MLE}
    \sqrt{k} \left( \hat{\gamma}_{k,n}^{\text{GPD}} - \gamma \right)
    \dto N \left( \lambda b(\gamma, \rho), (1 + \gamma)^2 \right),
    \qquad n \to \infty,
\end{equation}
where
\[
    b(\gamma, \rho) = \frac{\rho (1 + \gamma) (\gamma + \rho)}{\gamma (1 - \rho) (1 + \gamma - \rho)}.
\]
Comparing \eqref{E:AN:EPD} and \eqref{E:AN:MLE}, we see that if $\tau = -1$ and thus $\rho = - \gamma$, the asymptotic distributions of $\hat{\gamma}_{k,n}$ and $\hat{\gamma}_{k,n}^{\text{GPD}}$ coincide. This is in correspondance with the fact that the EGPD with $\tau = -1$ is a reparametrization of the GPD and the fact that the EPD estimators were obtained by solving the linearized score equations.

Finally, under the conditions of Theorem~\ref{T:estim}, the asymptotic distribution of the Hill estimator is
\begin{equation}
\label{E:AN:Hill}
    \sqrt{k} (H_{k,n} - \gamma)
    \dto N \left( \lambda \frac{\rho}{1-\rho}, \gamma^2 \right),
    \qquad n \to \infty;
\end{equation}
see for instance Theorem~\ref{T:simpler} below. Of the three estimators considered, the Hill estimator has the smallest asymptotic variance. Unless $\lambda = 0$, however, its asymptotic bias is never zero. The asymptotic distribution of the Hill estimator and its optimal variance property are of course well known; see for instance \citet[Section~9.4]{Reiss89}, \citet{Drees98} and \citet{BBW06}.

To illustrate the behavior of the three estimators, we generated samples from four different distributions. For each distribution, we generated $10,000$ samples of size $n = 1,000$ and computed the three extreme value index estimators for $k$ up to $500$. For the EPD estimator, we estimated the second-order parameter $\rho$ using the estimator in \citet{FAGdH03}. For each distribution and each estimator, we computed Monte Carlo estimates of the bias, variance and mean squared error by averaging out over the $10,000$ samples. 

Comparing the asymptotic results to the graphs in Figures~\ref{F:gamma:1}--\ref{F:gamma:2} we learn the following:
\begin{description}
\item[\it Fr\'echet distribution] with $\alpha = 1$. We have $\gamma = 1/\alpha = 1$, $\tau = -\alpha = -1$, and $\rho = \gamma \tau = -1$. From \eqref{E:AN:EPD} and \eqref{E:AN:MLE}, it follows that the asymptotic distributions of the EPD and the GPD estimators coincide, with zero asymptotic bias and an asymptotic variance of $4/k$. The Hill estimator has an asymptotic variance of $1/k$ only, but its asymptotic bias is nonzero. 

\item[\it Student \emph{t} distribution] with $\nu = 4$. We have $\gamma = 1/\nu = 1/4$, $\tau = -2$, and $\rho = \gamma \tau = -1/2$. The asymptotic variances of the three estimators are $\sigma^2 / k$ with $\sigma^2 = \gamma^2 = 1/16$ for the Hill estimator, $\sigma^2 = \gamma^2 (1 - \rho)^2 / \rho^2 = 9/16$ for the EPD estimator, and $\sigma^2 = (1 + \gamma)^2 = 25/16$ for the GPD estimator. Of the three estimators, the EPD estimator is the only one which is asymptotically unbiased. 

\item[\it Pareto mixture distribution] defined by $\bF(x) = (1+c)^{-1} x^{-\alpha} (1 + c x^{-\alpha})$, $x \geq 1$, with shape parameter $\alpha = 2$ and mixing parameter $c = 2$. We have $\gamma = 1/\alpha = 1/2$, $\tau = -\alpha = -2$, and $\rho = \gamma \tau = -1$. The weight of the second-order component is equal to $c = 2$ times the weight of the first-order component, inducing a severe bias to the Hill and GPD estimators; the EPD estimator is much less affected by this.  The asymptotic variances of the three estimators are $\sigma^2 / k$ with $\sigma^2 = \gamma^2 = 1/4$ for the Hill estimator, $\sigma^2 = \gamma^2 (1 - \rho)^2 / \rho^2 = 1$ for the EPD estimator, and $\sigma^2 = (1 + \gamma)^2 = 9/4$ for the GPD estimator. 

\item[\it Loggamma distribution] with shape parameter $\alpha = 4$ and scale parameter $\beta = 2$. Although this distribution has positive extreme-value index $\gamma = 1/\beta$, it is not in any of the classes ${\cal F}(\gamma, \tau)$, since $\bF(x) \sim \text{constant} \times x^{-1/\beta} (\log x)^{\alpha - 1}$. Nevertheless, the EPD estimator performs reasonably well when compared to the Hill and GPD estimators. 
\end{description}

\begin{figure}
\begin{center}
\begin{tabular}{cc}
\includegraphics[width = 0.45\textwidth, height = 0.28\textheight]{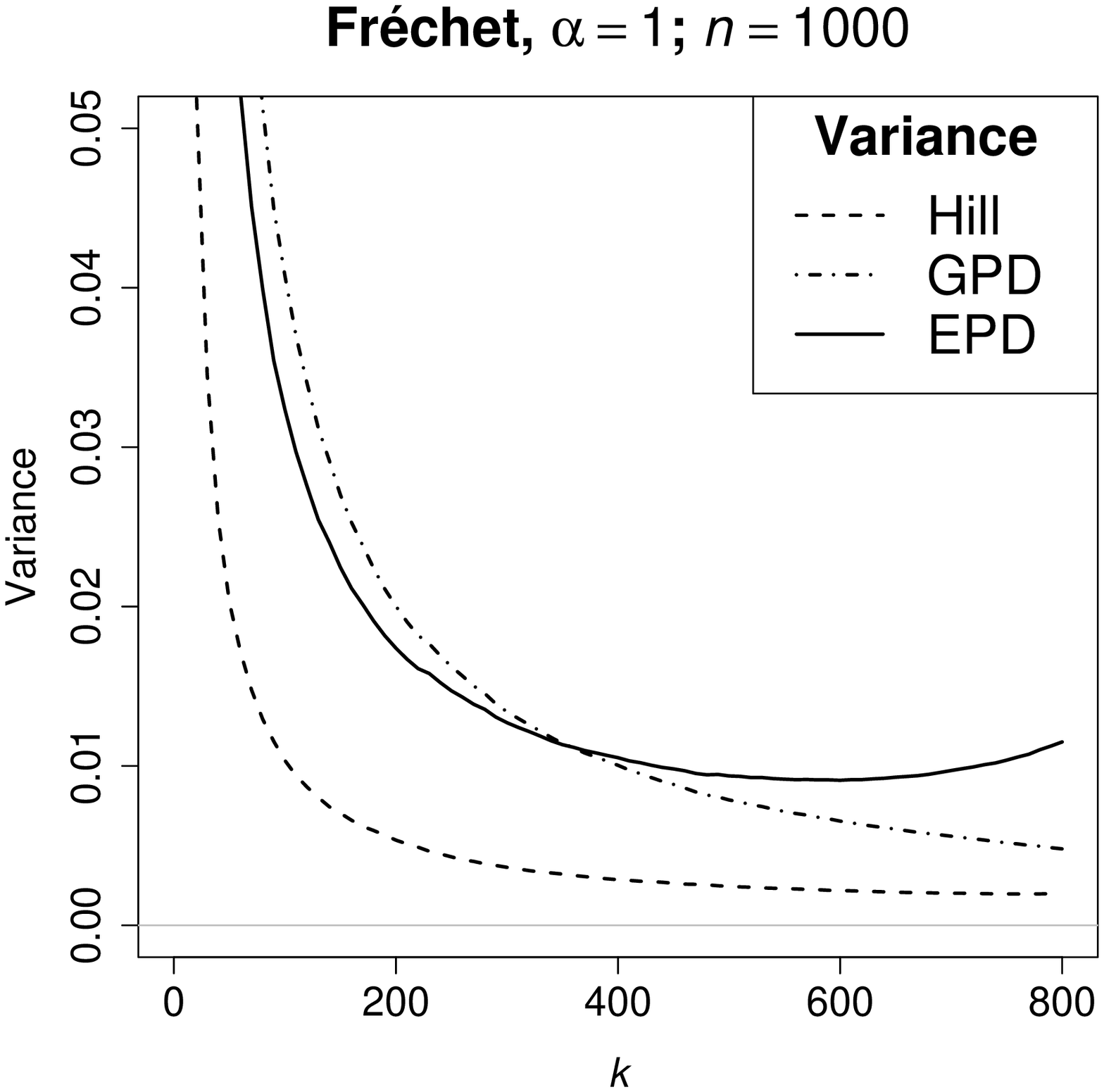} &
\includegraphics[width = 0.45\textwidth, height = 0.28\textheight]{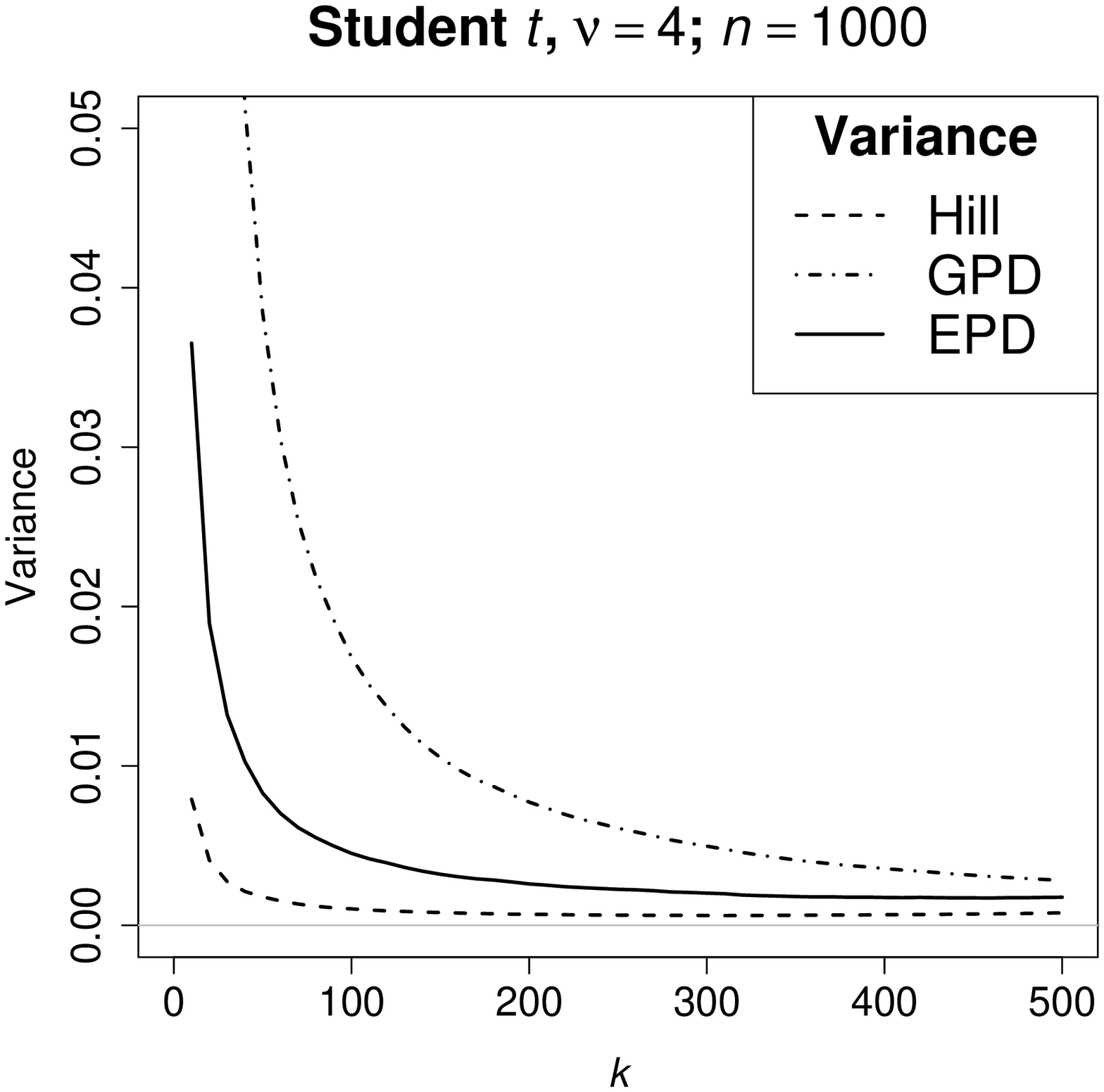} \\
\includegraphics[width = 0.45\textwidth, height = 0.28\textheight]{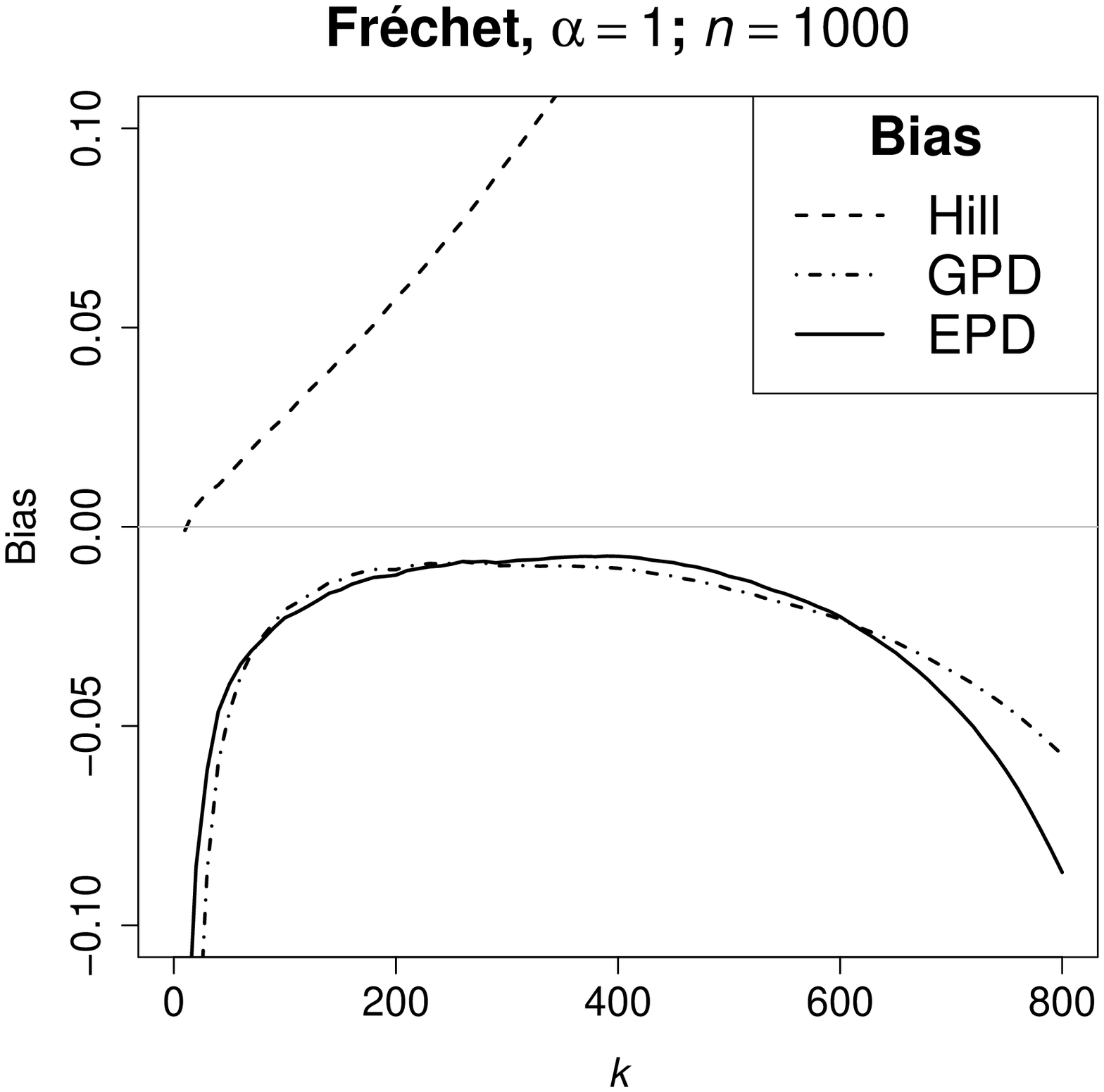} &
\includegraphics[width = 0.45\textwidth, height = 0.28\textheight]{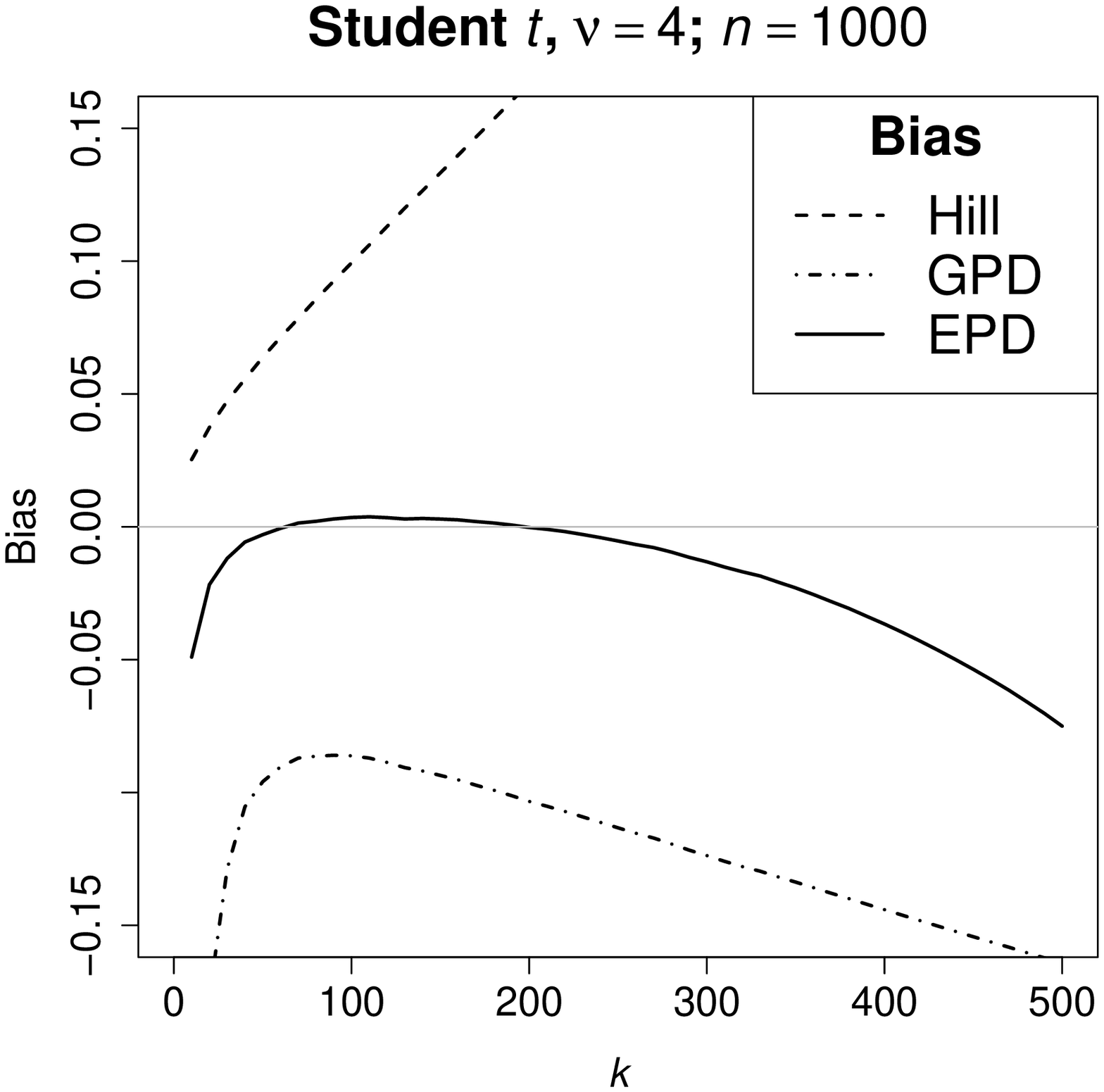} \\
\includegraphics[width = 0.45\textwidth, height = 0.28\textheight]{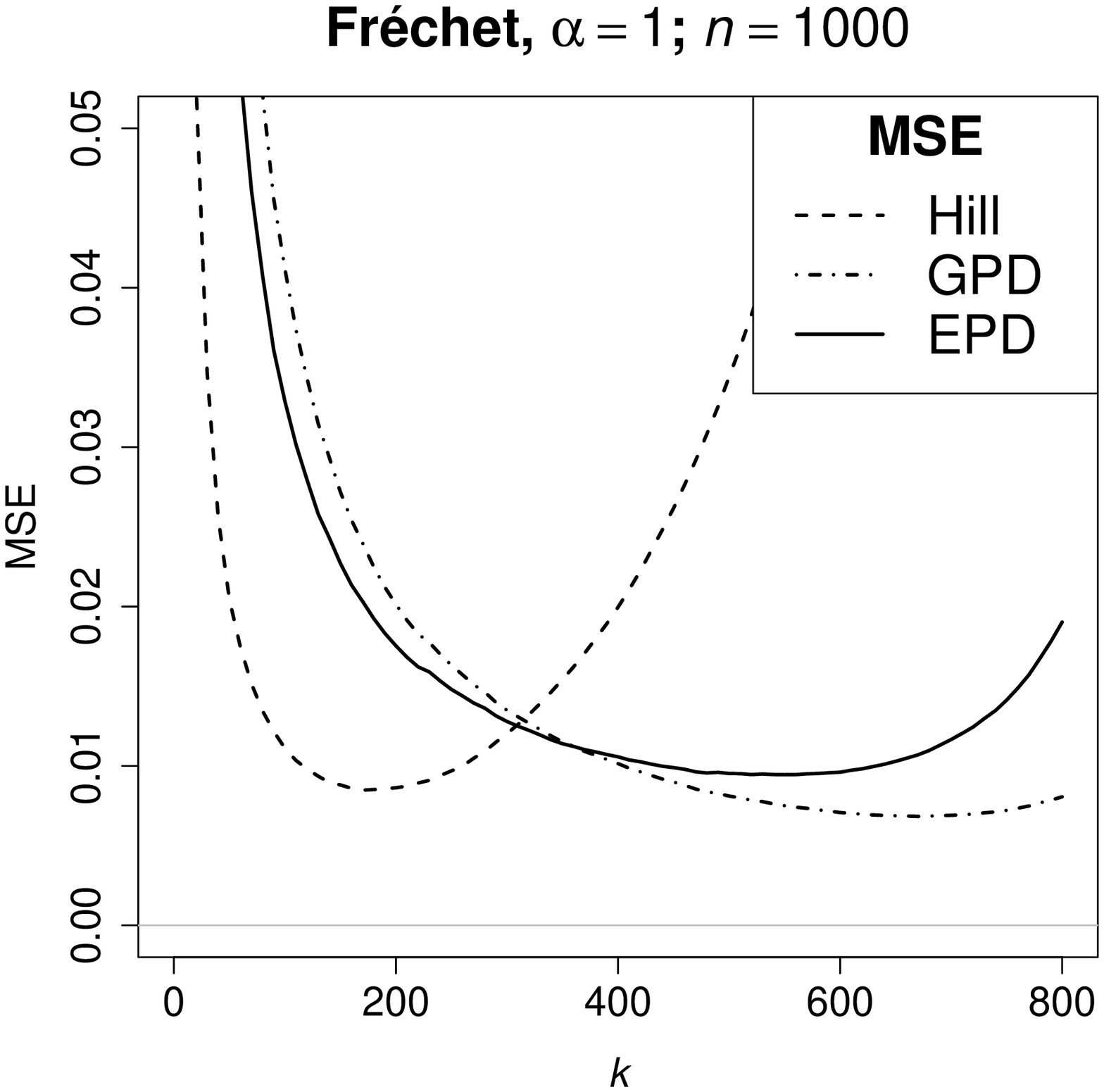} &
\includegraphics[width = 0.45\textwidth, height = 0.28\textheight]{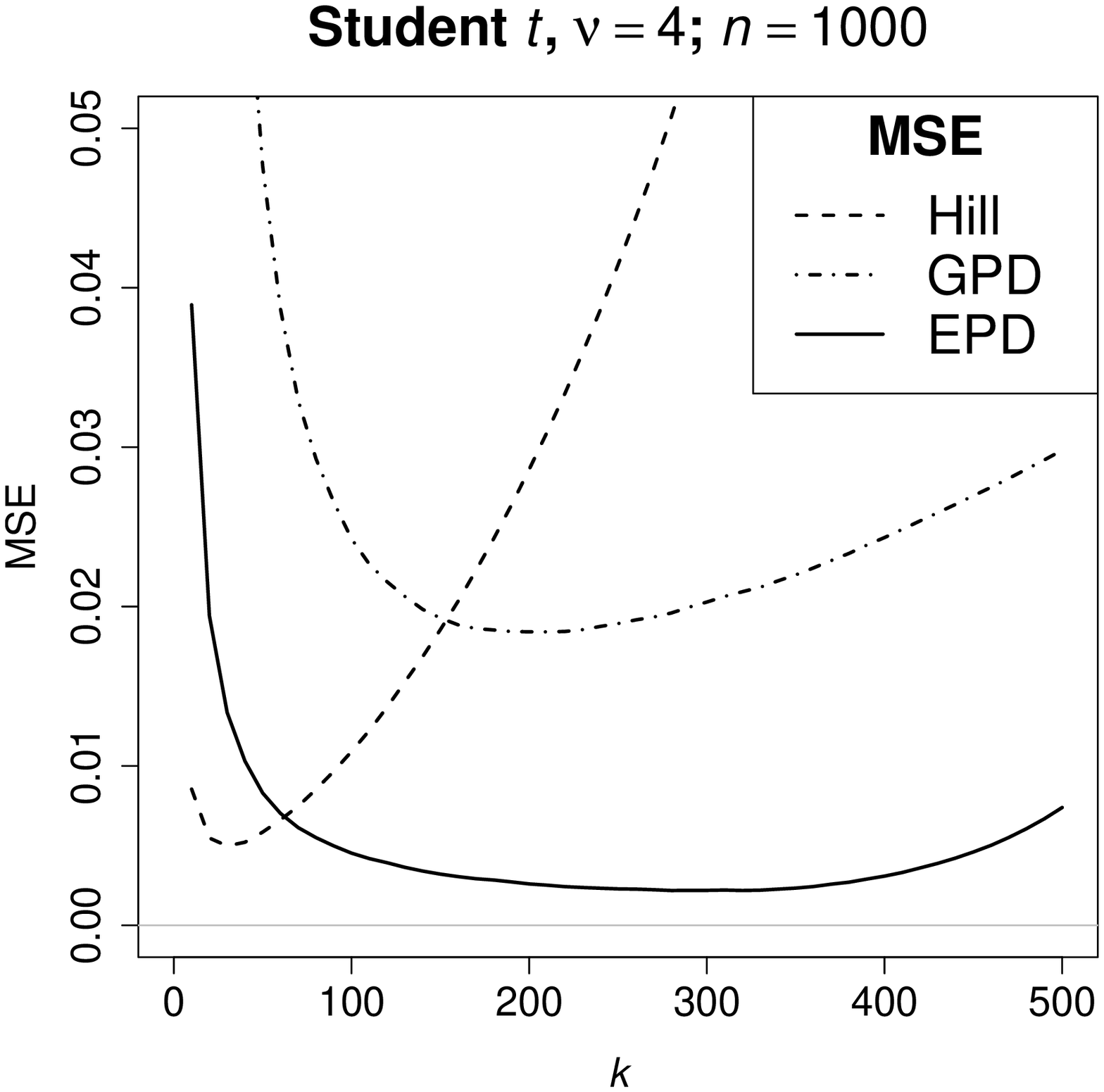}
\end{tabular}
\caption{\it Variance (top), bias (middle), and mean squared error
(bottom) of the Hill (dashed), GPD (dotdashed),
and EPD (solid) estimator in case of the unit Fr\'echet distribution
(left) and the Student t distribution with $\nu = 4$ degrees of
freedom (right). The sample size was $n = 1,000$ and the plots were
obtained by averaging out over $10,000$ samples.} \label{F:gamma:1}
\end{center}
\end{figure}

\begin{figure}
\begin{center}
\begin{tabular}{cc}
\includegraphics[width = 0.45\textwidth, height = 0.28\textheight]{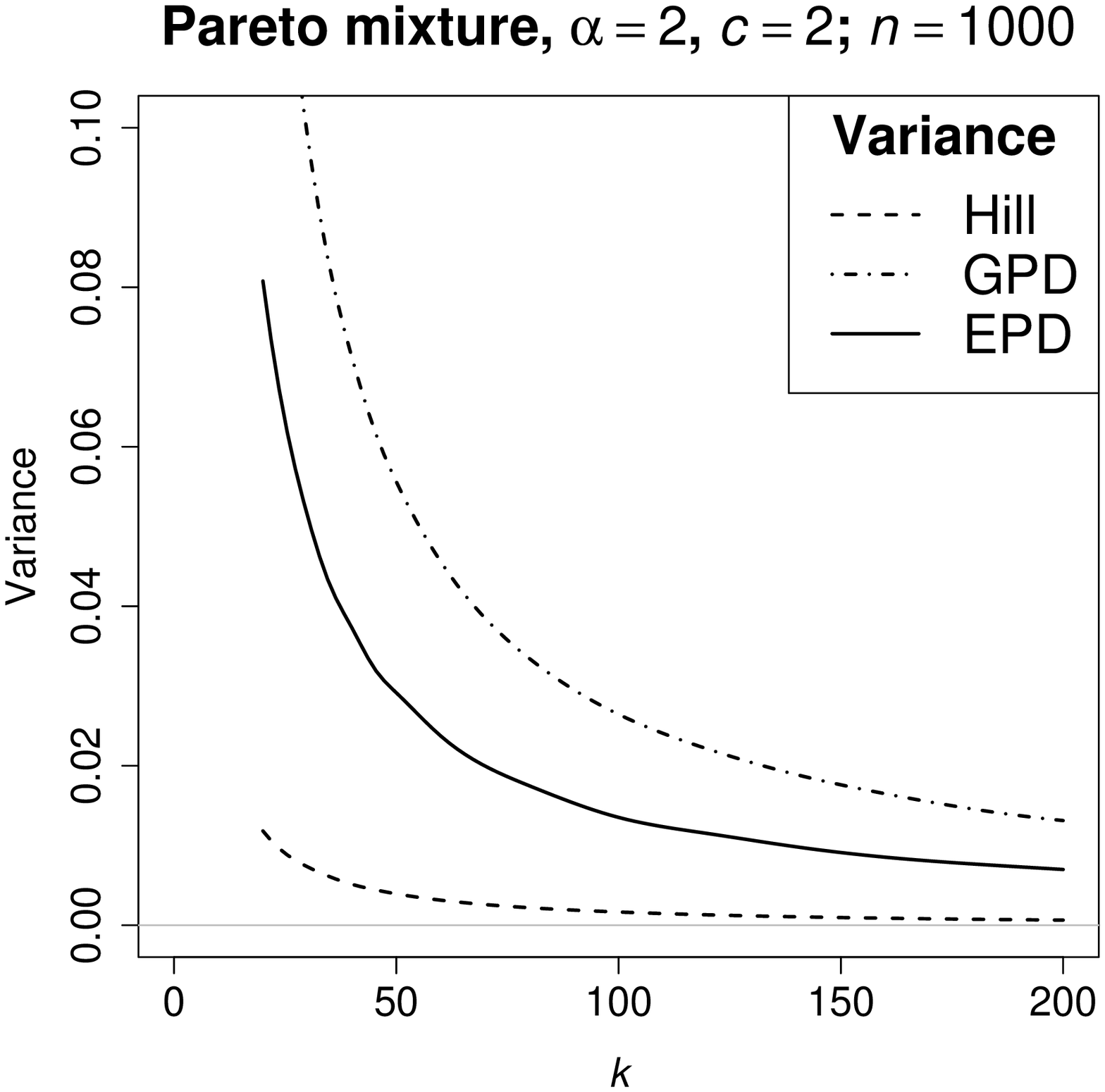} &
\includegraphics[width = 0.45\textwidth, height = 0.28\textheight]{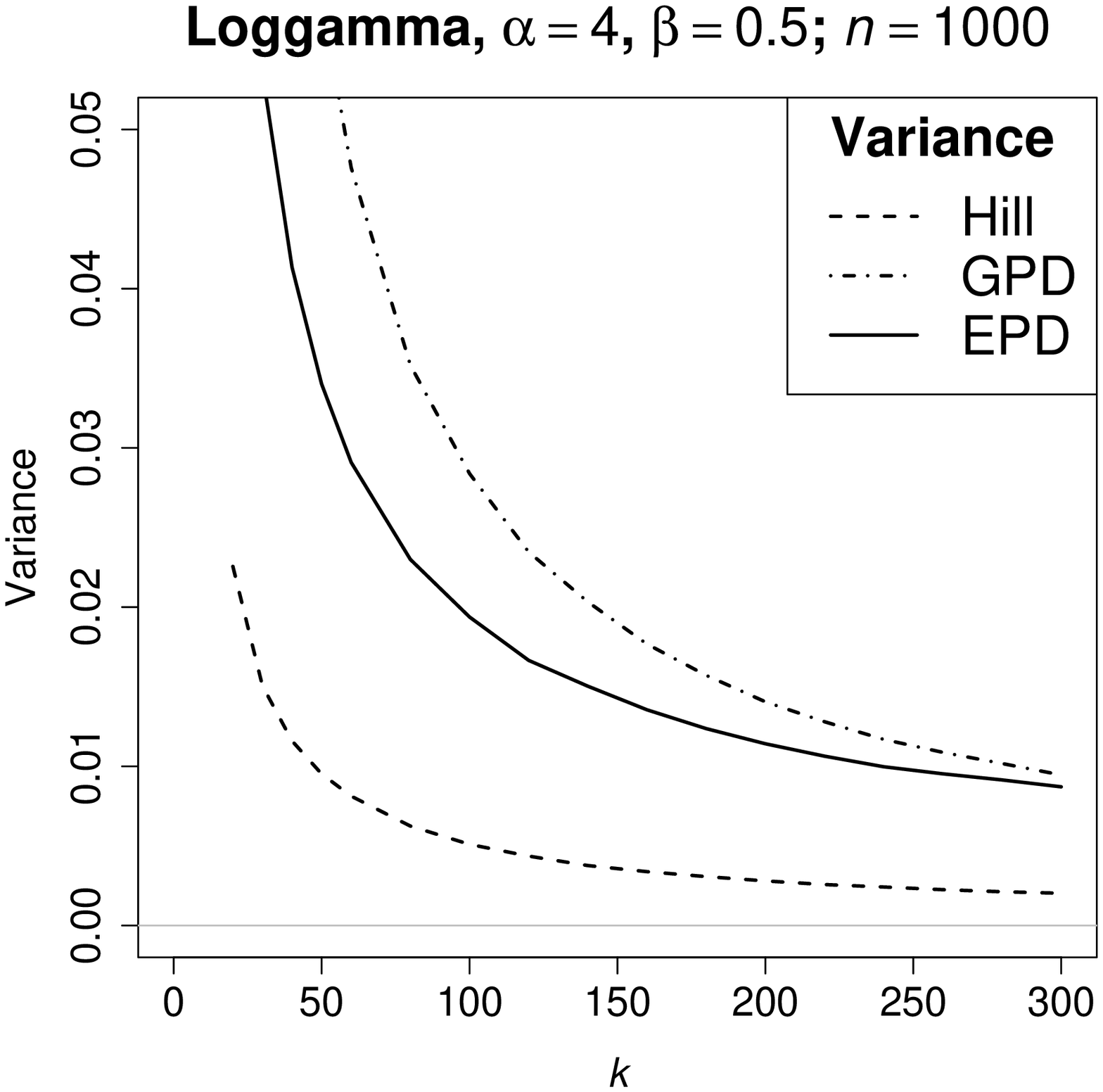} \\
\includegraphics[width = 0.45\textwidth, height = 0.28\textheight]{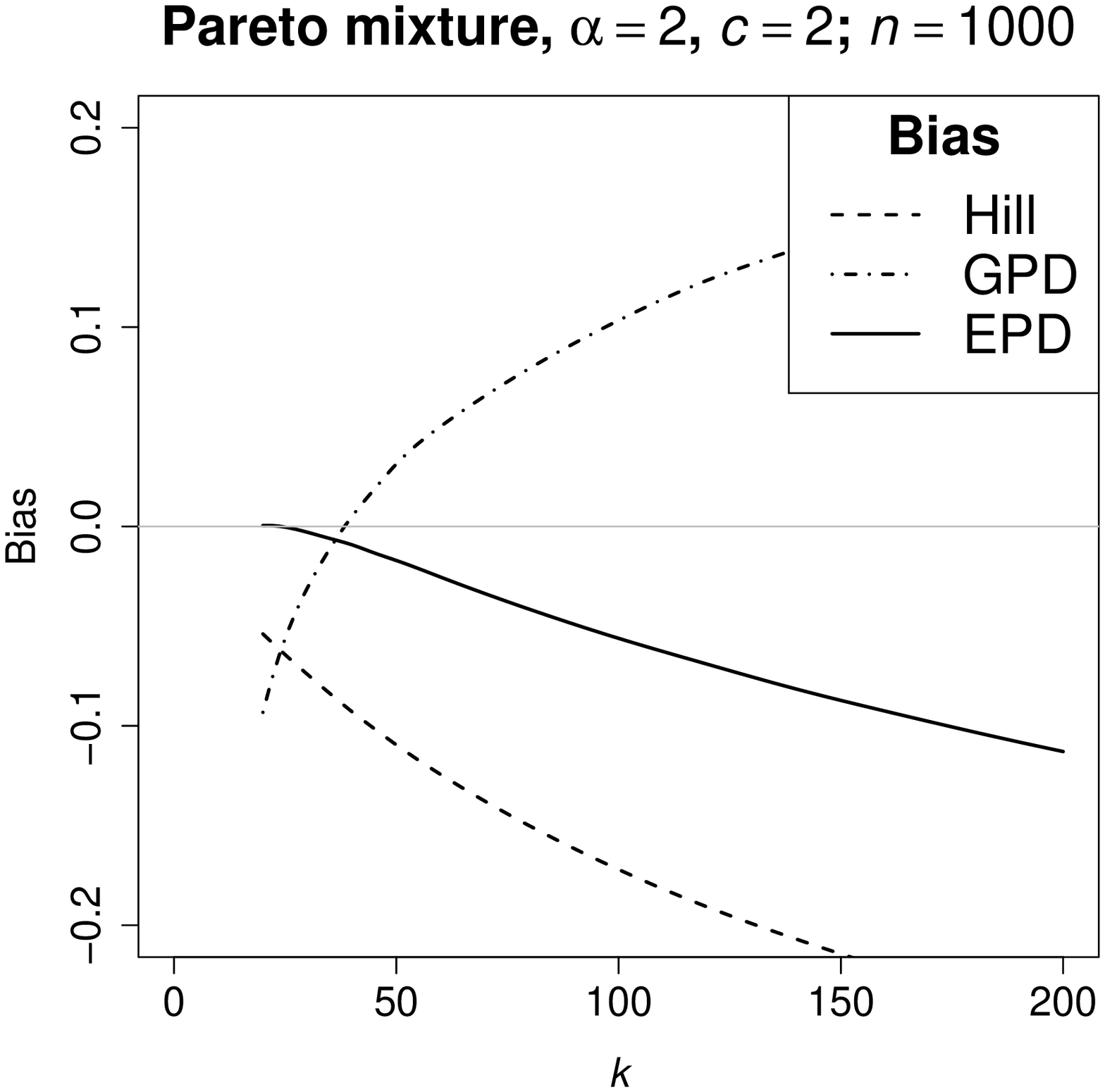} &
\includegraphics[width = 0.45\textwidth, height = 0.28\textheight]{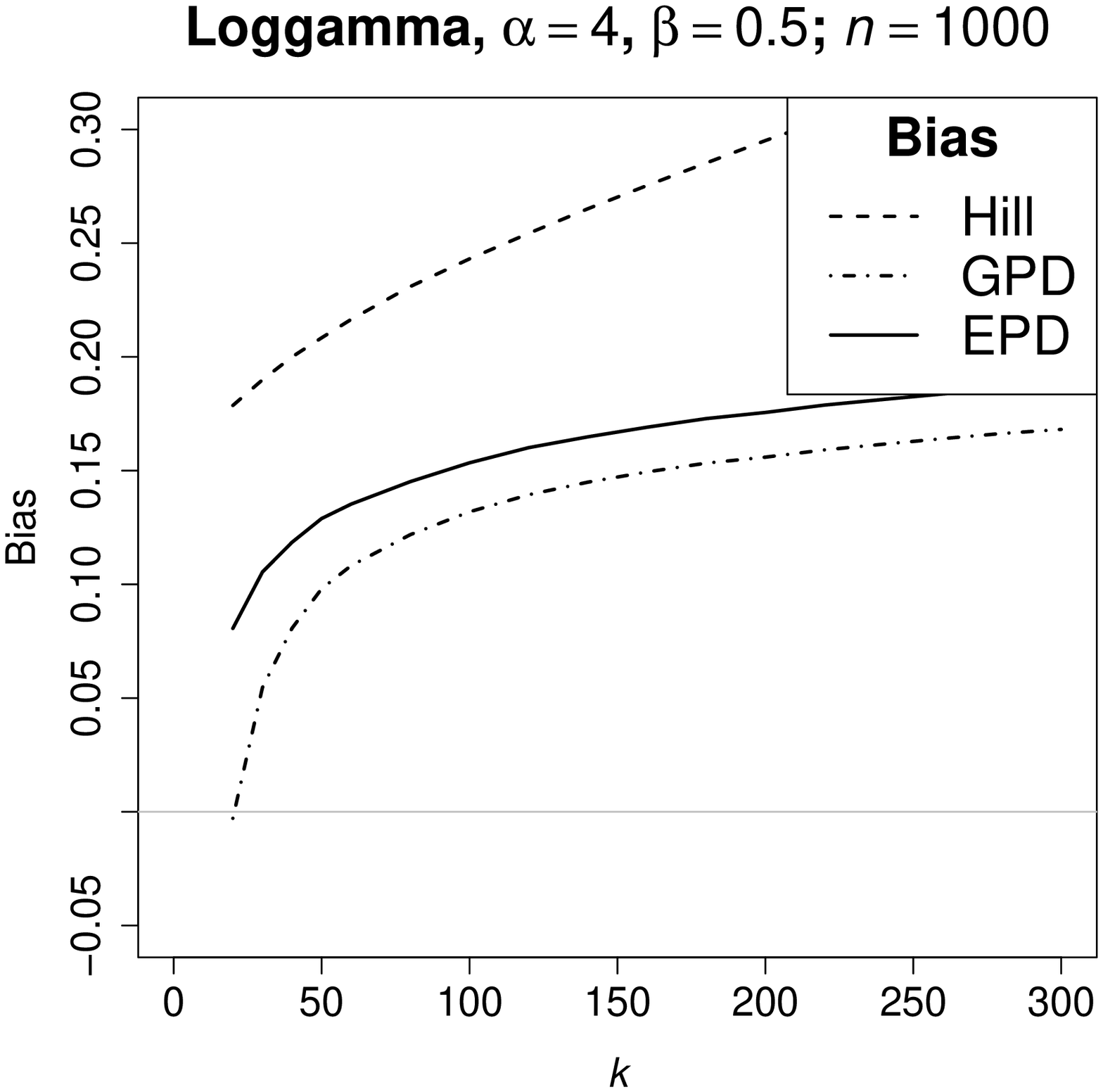} \\
\includegraphics[width = 0.45\textwidth, height = 0.28\textheight]{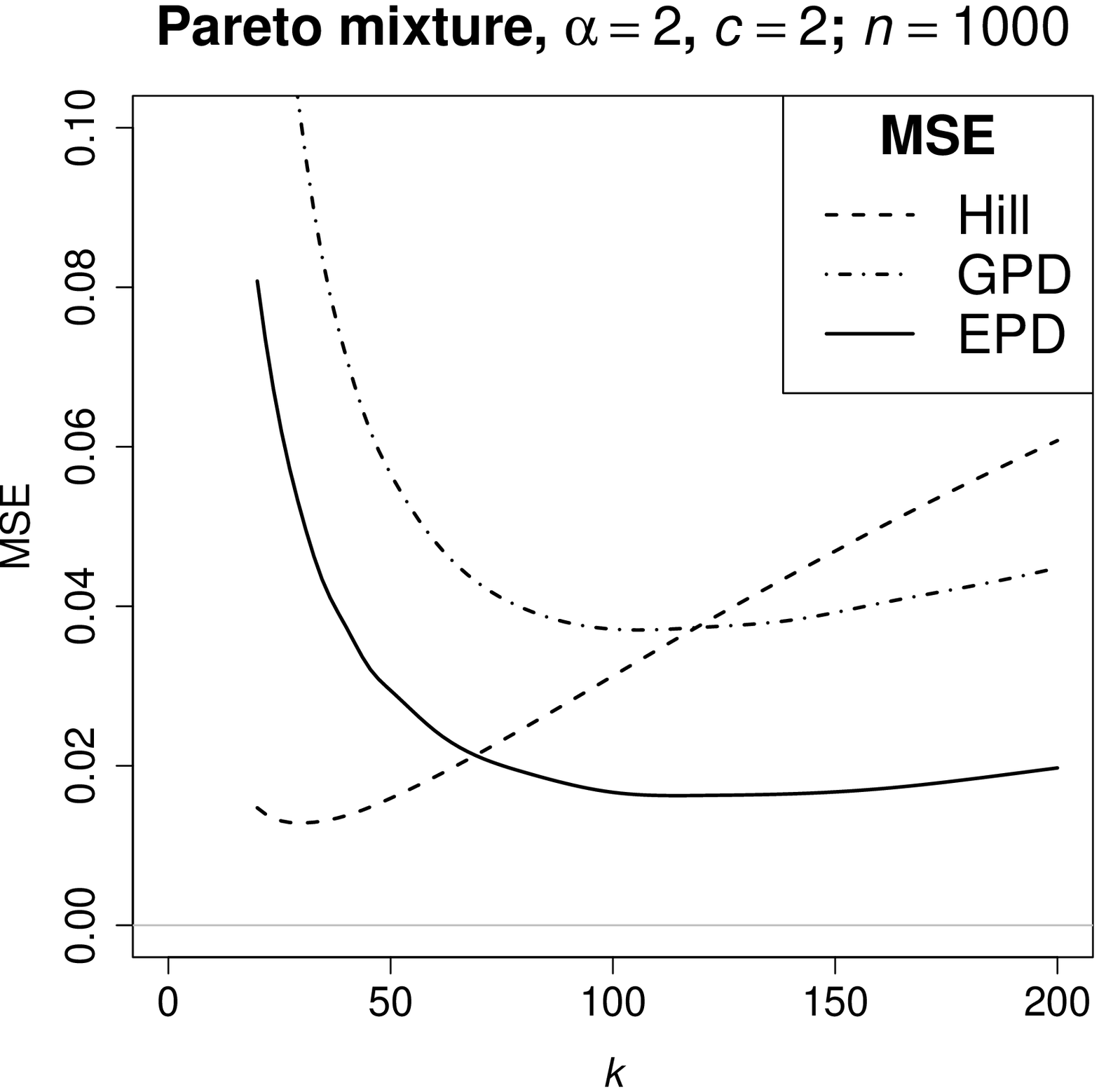} &
\includegraphics[width = 0.45\textwidth, height = 0.28\textheight]{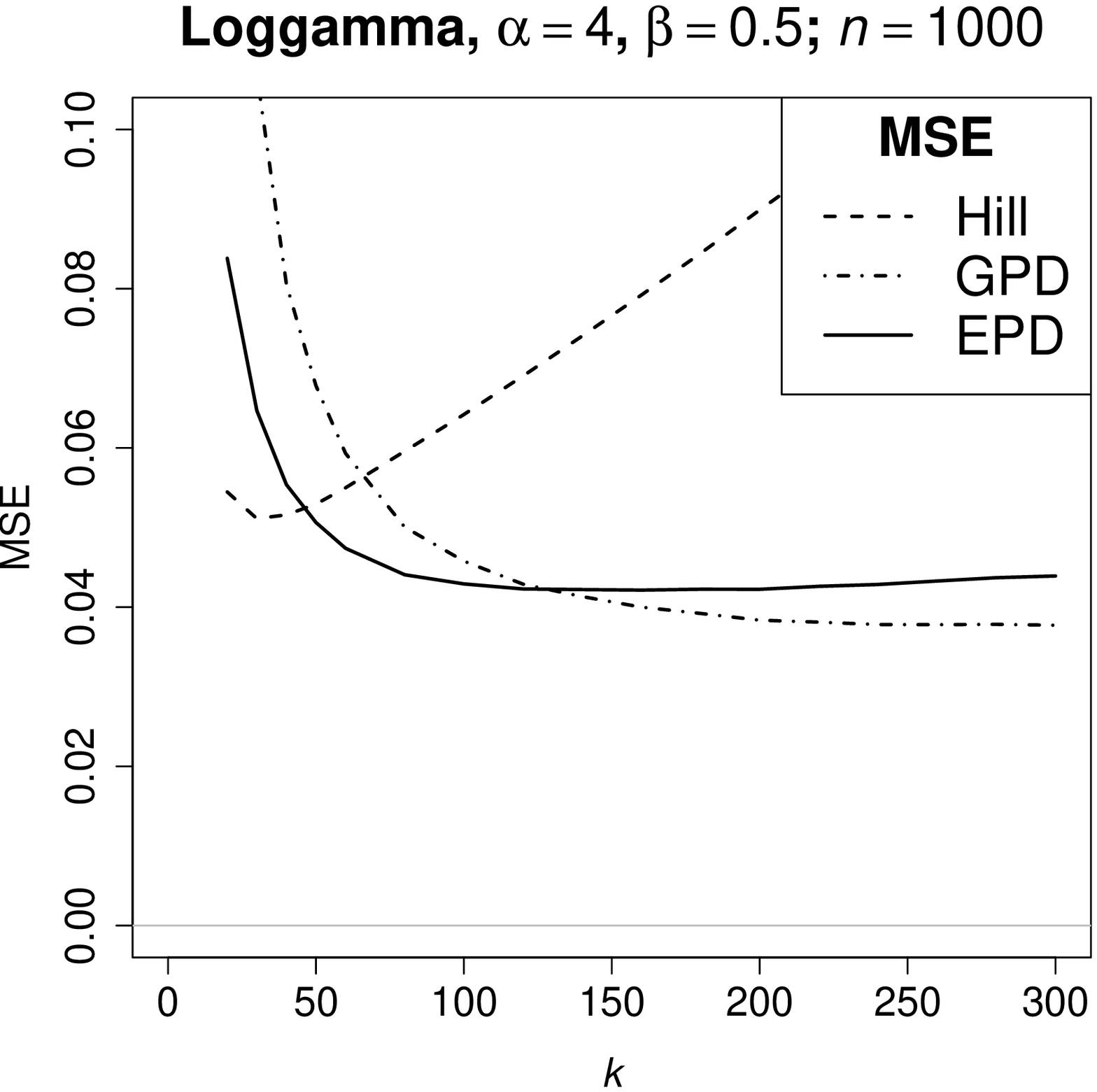}
\end{tabular}
\caption{\it Variance (top), bias (middle), and mean squared error (bottom) of the Hill (dashed), GPD (dotdashed), and EPD (solid) estimator in case of a Pareto mixture distribution ($\alpha = 2$, $c = 2$; left) and a Loggamma distribution ($\alpha = 4$, $\beta = 2$; right). The sample size was $n = 1,000$ and the plots were obtained by averaging out over $10,000$ samples.} \label{F:gamma:2}
\end{center}
\end{figure}

\section{Tail Probability Estimation}
\label{S:prob}

Let us return to the tail estimation problem raised in the beginning
of Section~\ref{S:par}. Given the order statistics $X_{1:n} \leq
\cdots \leq X_{n:n}$ of an independent sample from an unknown
distribution function $F \in {\cal F}(\gamma, \tau)$, we want to
estimate the tail probability $p_n = \bF(x_n)$, where $x_n \to
\infty$ and thus $p_n \to 0$ as $n \to \infty$. As before, let $k =
k_n \in \{1, \ldots, n-1\}$ be an intermediate integer sequence,
that is, $k \to \infty$ and $k/n \to 0$. Assume that $p_n =
\bF(x_n)$ satisfies
\begin{equation}
\label{E:pn}
    n p_n / k \to q \in [0, 1), \qquad n \to \infty.
\end{equation}
Let $\hat{\gamma}_n$, $\hat{\delta}_n$, and $\hat{\tau}_n$ denote general estimator sequences and put
$\delta_n = \delta(X_{n-k:n})$ as well as
\begin{equation}
\label{E:GammaDelta}
    \Gamma_{k,n} = \sqrt{k} (\hat{\gamma}_n - \gamma)
    \qquad \text{and} \qquad
    \Delta_{k,n} = \sqrt{k} (\hat{\delta}_n - \delta_n).
\end{equation}
Recall $Z_{k,n}$ in \eqref{E:Zkn} and assume that
\begin{equation}
\label{E:joint}
    \hat{\tau}_n = \tau + o_p(1)
    \quad \text{and} \quad
    (\Gamma_{k,n}, \Delta_{k,n}, Z_{k,n}) \dto (\Gamma, \Delta, Z), \qquad n \to \infty,
\end{equation}
a trivariate random vector. A possible choice for the estimators of $\gamma$ and
$\delta_n$ are the ones studied in Theorem~\ref{T:estim}. However,
we will formulate our results so as to allow for general estimator
sequences satisfying \eqref{E:joint}. For the estimator of $\tau$,
one can for instance take $\hat{\tau}_n = \hat{\rho}_n /
\hat{\gamma}_n$, where $\hat{\rho}_n$ is an estimator of $\rho =
\gamma \tau$, see for instance \citet{FAGdH03}. As in Theorem~\ref{T:estim}, the
asymptotic distribution of $\hat{\tau}_{n}$ plays no role.

Omitting the remainder term in \eqref{E:approx} and replacing the
unknown quantities $\bF(u)$ and $(\gamma, \delta(u), \tau)$ at the
random threshold $u = X_{n-k:n}$ by $k/n$ and $(\hat{\gamma}_n,
\hat{\delta}_n, \hat{\tau}_n)$, respectively, yields the estimator
\[
    \hat{p}_{k,n} = \hat{\bF}_n(x_n) = \frac{k}{n} \bG_{\hat{\gamma}_n, \hat{\delta}_n, \hat{\tau}_n}(x_n / X_{n-k:n}).
\]
In the same way, one can construct estimators for other tail quantities: return levels,
expected shortfall, etc. For brevity, we focus here on tail
probabilities.

In order to describe the asymptotics of $\hat{p}_{k,n}$, we need
to make a distinction between the case $0 < q < 1$ in \eqref{E:pn}
and $q = 0$. The proofs of the following two theorems are to be found
in Appendix~\ref{A:tail}. Results for tail probability estimators based on the PD and GPD approximations can be found in \citet[Section~4.4]{dHF}.

\begin{theorem}
\label{T:tail}
Let $F \in {\cal F}(\gamma, \tau)$, let $k_n$ be an intermediate sequence satisfying \eqref{E:k} and let $p_n$ be such that \eqref{E:pn} holds for some $0 < q < 1$. If \eqref{E:joint}, then
\begin{equation}
\label{E:tail}
    \sqrt{k} \left( \frac{\hat{p}_{k,n}}{p_n} - 1 \right)
    \dto - \gamma^{-1} \Gamma \log q - \gamma^{-1} \Delta (1 - q^{-\rho}) - Z, \qquad n \to \infty.
\end{equation}
\end{theorem}

\begin{theorem}
\label{T:tail:extreme}
In Theorem~\ref{T:tail}, if \eqref{E:pn} is replaced by
\[
    np_n / k \to 0 \qquad \text{and} \qquad \log (np_n) / \sqrt{k} \to 0, \qquad n \to \infty,
\]
then
\[
    \frac{\sqrt{k}}{\log \{ k / (np_n) \}} \left( \frac{\hat{p}_{k,n}}{p_n} - 1 \right)
    \dto \gamma^{-1} \Gamma, \qquad n \to \infty.
\]
\end{theorem}

For the EPD estimators $\hat{\gamma}_n = \hat{\gamma}_{k,n}$ and $\hat{\delta}_n = \hat{\delta}_{k,n}$, Theorems~\ref{T:estim} and \ref{T:tail} lead to
\begin{equation}
\label{E:pn:EPD}
    \sqrt{k} \left( \frac{\hat{p}_n}{p_n} - 1 \right)
    \dto N \left(0, \sigma^2(q, \rho) \right), \qquad n \to \infty,
\end{equation}
with asymptotic variance given by
\begin{multline*}
    \sigma^2(q, \rho)
    = (\log q)^2 \frac{(1-\rho)^2}{\rho^2} + \left( \frac{1-q^{-\rho}}{\rho} \right)^2 \frac{(1-2\rho)(1-\rho)^2}{\rho^2} \\
    - 2 \log (q) \frac{1 - q^{-\rho}}{\rho} \frac{(1-2\rho)(1-\rho)}{\rho^2} + 1.
\end{multline*} 
The importance of the fact that the limit distribution in \eqref{E:pn:EPD} has mean zero was already discussed after Theorem~\ref{T:estim}. An asymptotic confidence interval of nominal level $1 - \alpha$ is given by
\begin{equation}
\label{E:CI:p}
	\biggl[
		\hat{p}_n \biggl( 1 - \sigma(\hat{q}_n, \hat{\rho}_n) \frac{z_{\alpha/2}}{\sqrt{k}} \biggr), \;
		\hat{p}_n \biggl( 1 + \sigma(\hat{q}_n, \hat{\rho}_n) \frac{z_{\alpha/2}}{\sqrt{k}} \biggr)
	\biggr]
\end{equation}
where $\hat{q}_n = n \hat{p}_n / k_n$ and with $z_{\alpha/2}$ the $1-\alpha/2$ quantile of the standard normal distribution.

If we simply define $\hat{\delta}_n = 0$, then $\Delta_{k,n} = - \sqrt{k} \delta_n$ in \eqref{E:GammaDelta} and thus $\Delta = -\lambda$ in \eqref{E:joint}. The tail probability estimator $\hat{p}_n$ then reduces to the Weissman estimator \citep{Weissman78}
\begin{equation}
\label{E:Weissman}
    \hat{p}_n^{\mathrm{W}} = \frac{k}{n} \left( \frac{x_n}{X_{n-k:n}} \right)^{-1/\hat{\gamma}_n}.
\end{equation}
Theorem~\ref{T:tail} then implies
\begin{equation}
\label{E:AN:Weissman}
    \sqrt{k} \left( \frac{\hat{p}_n^{\mathrm{W}}}{p_n} - 1 \right)
    \dto - \gamma^{-1} \Gamma \log q + \gamma^{-1} \lambda (1 - q^{-\gamma \tau}) - Z,
    \qquad n \to \infty.
\end{equation}
For instance, if we estimate $\gamma$ by the Hill estimator, then in view of Theorem~\ref{T:simpler},
\begin{multline*}
    \sqrt{k} \left( \frac{k}{np_n} \left( \frac{x_n}{X_{n-k:n}} \right)^{-1/H_{k,n}} - 1 \right) \\
    \dto N \left( -\lambda \frac{\rho}{\gamma} \left( \frac{q^{-\rho}-1}{\rho} + \frac{\log q}{1 - \rho}\right), 1 + (\log q)^2 \right),
    \qquad n \to \infty.
\end{multline*}

Even if the extreme value index estimator $\hat{\gamma}_n$ is such that the asymptotic distribution of $\Gamma_n$ has mean zero, then still the asymptotic distribution \eqref{E:AN:Weissman} of the Weissman estimator will have a mean which is proportional to $\lambda$. In other words, unbiased tail estimation requires more than unbiased estimation of the extreme value index alone.

From Theorem~\ref{T:tail:extreme} and its proof, we learn that for estimation of tail probabilities $p_n$ of smaller order than $k/n$, the difference between the Pareto approximation and the EPD approximation does not matter asymptotically. Still, for $\hat{p}_n$ to be an asymptotically unbiased estimator of $p_n$, the estimator $\hat{\gamma}_n$ needs to be asymptotically unbiased for $\gamma$. For instance, if we use the EPD estimator $\hat{\gamma}_{k,n}$, then
\[
    \frac{\sqrt{k}}{\log \{ k / (np_n) \}} \left( \frac{\hat{p}_n}{p_n} - 1 \right)
    \dto N \left( 0, \frac{(1 - \rho)^2}{\rho^2} \right),
    \qquad n \to \infty.
\]

\begin{example}\upshape
\label{Ex:secura}
The Secura Belgian Re data in \citet[Section~1.3.3]{BGST} comprise 371 automobile claims not smaller than \euro 1.2 million. The data span the period 1988--2001 and have been gathered from several European insurance companies. Figure~\ref{F:secura} shows the estimates of $\gamma$ (left) and of the probability of a claim to exceed \euro 7 million (right). Nominal 90 \% confidence intervals for the EPD estimates are added too, see \eqref{E:CI:gamma} and \eqref{E:CI:p}. In the data-set, there were actually 3 exceedances over \euro 1.2 million, yielding a nonparametric estimate of $3/371 = 0.81 \%$. In comparison to the Weissman (Hill) and POT (GPD) estimates, the trajectories of the EPD estimates are relatively stable, with $\hat{\gamma}$ around $0.3$ and $p$ around $0.75 \%$. By way of comparison, in \citet[Section~6.2.4]{BGST} it is suggested to model the complete distribution by a mixture of two components, an exponential and a Pareto distribution, with the knot at about \euro 2.6 million, which corresponds to the order statistic $X_{n-k:n}$ with $k = 95$. Although this knot is detected by the EPD estimator, it does not cause the tail parameter estimates to change dramatically.
\begin{figure}
\begin{center}
\begin{tabular}{cc}
\includegraphics[width = 0.45\textwidth, height = 0.29\textheight]{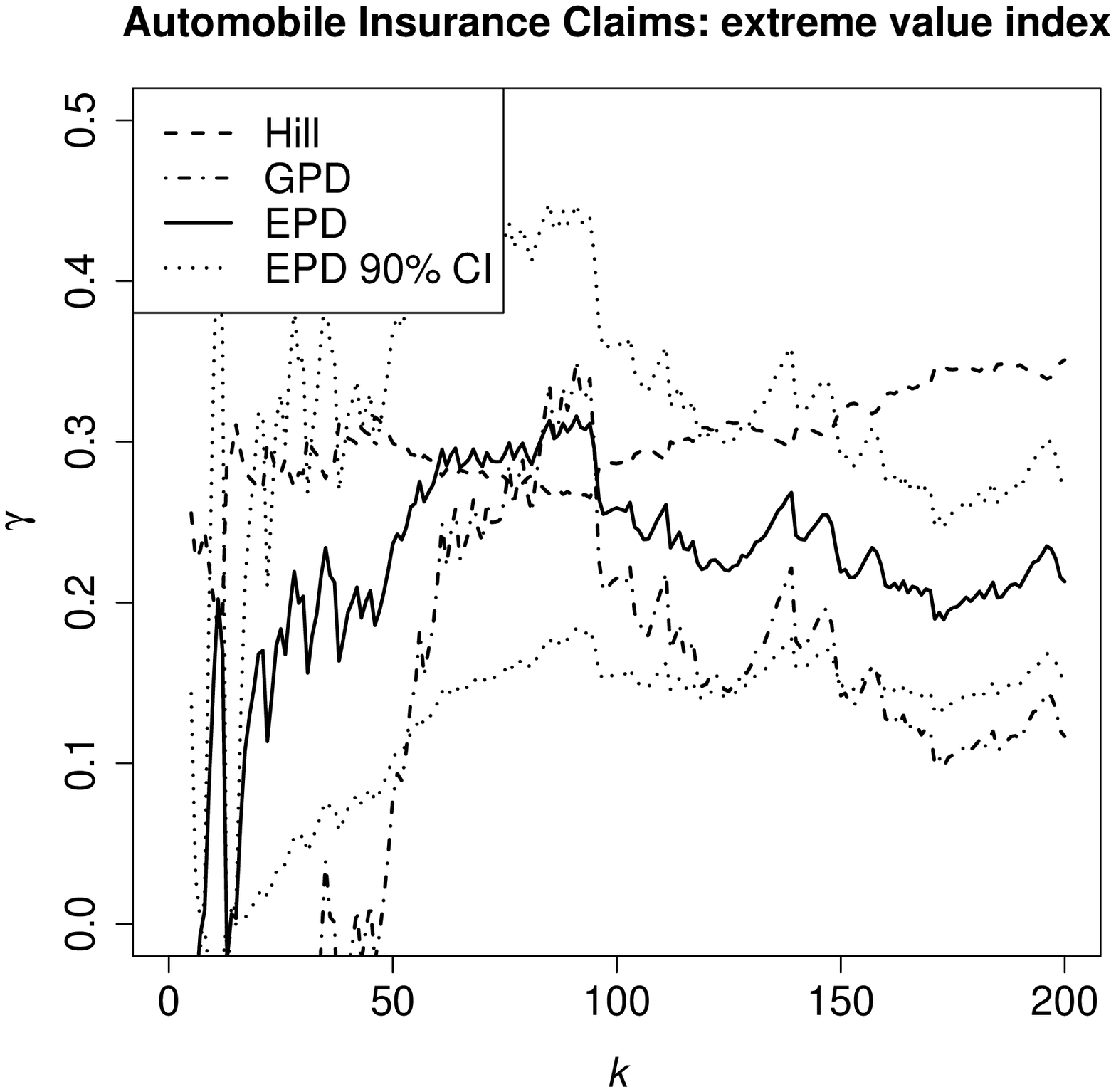} &
\includegraphics[width = 0.45\textwidth, height = 0.29\textheight]{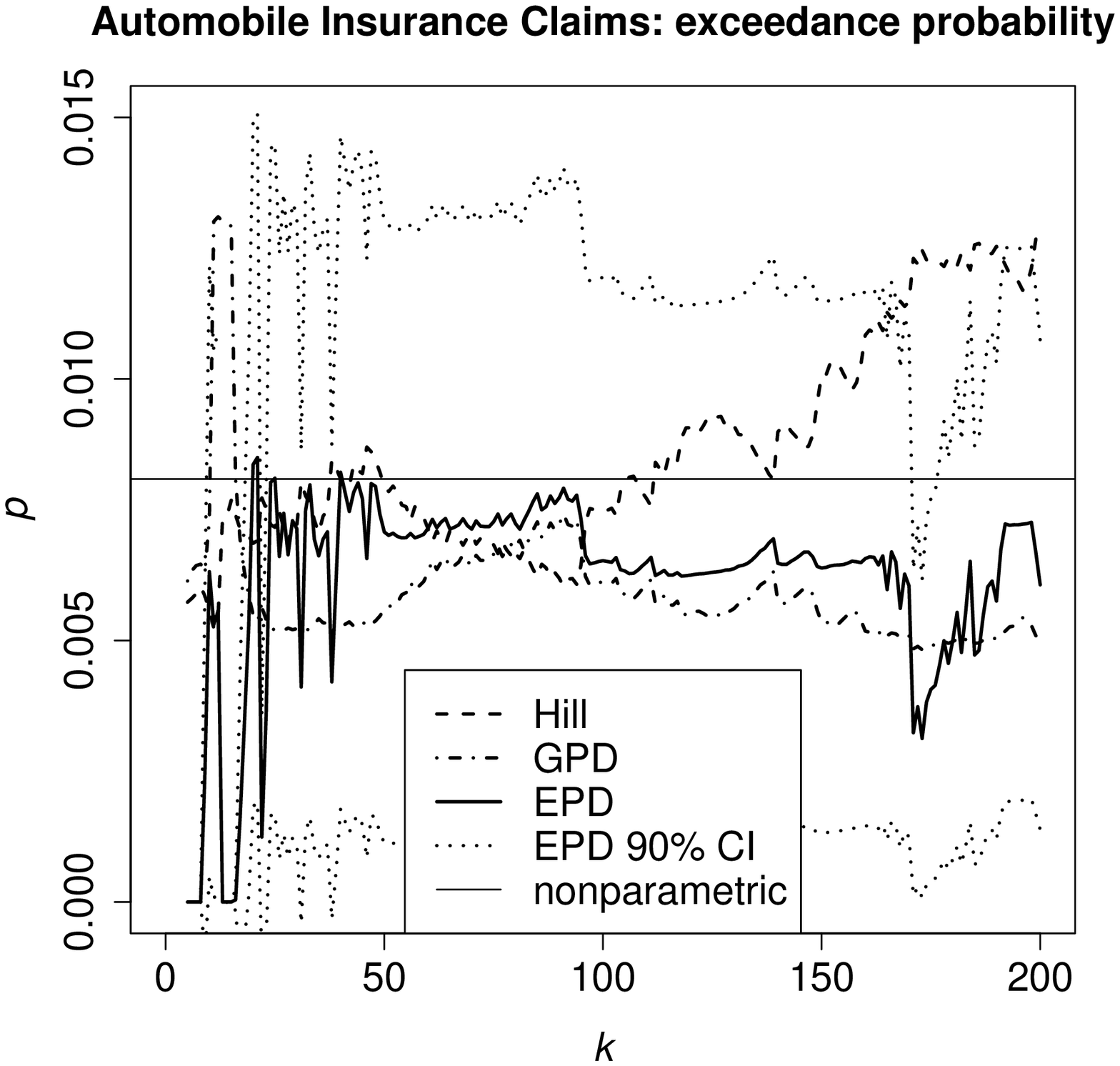}
\end{tabular}
\caption{\it Trajectories of estimates of $\gamma$ (left) and of the exceedance probability over \euro 7 million (right) for the Secura Belgian Re data in Example~\ref{Ex:secura}. \label{F:secura}}
\end{center}
\end{figure}
\end{example}

\appendix
\section{Tail Empirical Processes}
\label{A:simpler}

Recall $H_{k,n}$, $E_{k,n}(s)$ and $Z_{k,n}$ from equations \eqref{E:Hill}, \eqref{E:Ekn} and \eqref{E:Zkn}, respectively, and define
\begin{align}
\label{E:Gammakn}
    \Gamma_{k,n} &= \sqrt{k} (H_{k,n} - \gamma), \\
\label{E:EEkn}
    \mathbb{E}_{k,n}(s) &= \sqrt{k} \left( E_{k,n}(s) - \frac{1}{1 - s \gamma} \right), \qquad s \leq 0.
\end{align}
Our proof of Theorem~\ref{T:estim} will be based on the fact that $(\Gamma_{k,n}, \mathbb{E}_{k,n}, Z_{k,n})$
converges weakly in the space $\RR \times {\cal C}[s_0, 0] \times
\RR$; here $s_0 < 0$ and ${\cal C}[a, b]$ is the Banach space of
continuous functions $f : [a, b] \to \RR$ equipped with the topology
of uniform convergence. Of course, the asymptotic distribution of the normalized Hill estimator $\Gamma_{k,n}$ has been established in numerous other papers; in the following theorem, it is the joint convergence which is our main concern.

\begin{theorem}
\label{T:simpler}
Let $F \in {\cal F}(\gamma, \tau)$. If $k = k_n$ is an intermediate integer sequence satisfying \eqref{E:k}, then for every $s_0 < 0$, in $\RR \times {\cal C}[-s_0, 0] \times \RR$,
\[
    (\Gamma_{k,n}, \EE_{k,n}, Z_{k,n}) \dto (\Gamma, \EE, Z), \qquad n \to \infty,
\]
a Gaussian process with the following distribution: $Z$ is standard normal and is independent of $(\Gamma, \EE)$, and for $s, s_1, s_2 \in [s_0, 0]$,
\begin{align*}
    \E [ \EE (s) ] &= \lambda \frac{s \rho}{(1 - s \gamma - \rho)(1 - s \gamma)}, 
			& \E [ \Gamma ] &= \lambda \frac{\rho}{1-\rho}, \\
    \cov \{ \EE(s_1), \EE(s_2) \} &= \frac{s_1 s_2 \gamma^2}{(1 - s_1 \gamma - s_2 \gamma)(1 - s_1 \gamma)(1 - s_2 \gamma)}, 
			& \var ( \Gamma ) &= \gamma^2, \\
    \cov \{ \Gamma, \EE(s) \} &= \frac{s \gamma^2}{(1 - s \gamma)^2}.
\end{align*}
\end{theorem}

\begin{proof}
Let $Y_1, \Y_1, Y_2, \Y_2, \ldots$ be independent Pareto(1) random
variables. For positive integer $k$, denote the order statistics of
$Y_1, \ldots, Y_k$ by $Y_{1:k} < \cdots < Y_{k:k}$; also, let
$Y_{0:k} = 1$. Similarly, denote the order statistics of $\Y_1,
\ldots, \Y_n$ by $\Y_{1:n} < \cdots < \Y_{n:n}$. Then the following
three vectors are equal in distribution:
\begin{align*}
    ( X_{n-k+i:n} : i = 0, \ldots, k)
    &\eqd ( U(\Y_{n-k+i:n}) : i = 0, \ldots, k ) \\
    &\eqd ( U(Y_{i:k} \Ynk) : i = 0, \ldots, k ).
\end{align*}
Since we are only interested in the asymptotic distribution of $(\Gamma_{k,n}, \EE_{k,n}, Z_{k,n})$, we may without loss of generality assume that actually
\[
    ( X_{n-k+i:n} : i = 0, \ldots, k)
    = ( U(Y_{i:k} \Ynk) : i = 0, \ldots, k ).
\]
The following property is well-known: if $k$ is an intermediate sequence, then
\begin{equation}
\label{E:Ynk}
    \sqrt{k} \{ (n/k) \Ynk^{-1} - 1 \} \dto N(0, 1), \qquad n \to \infty.
\end{equation}
[A quick proof is to employ the distributional representation $\Ynk \eqd (E_1 + \cdots + E_{n+1}) / (E_1 + \ldots + E_k)$, with $E_1, \ldots, E_n$ independent standard exponential random variables.] As a consequence, we have $\Ynk = (n/k) \{1 + o_p(1)\}$ as $n \to \infty$, and therefore, by \eqref{E:k} and the Uniform Convergence Theorem for $\RV_\rho$ \citep[Theorem~1.5.2]{BGT},
\begin{equation}
\label{E:aYnk}
    \sqrt{k} a(\Ynk)
    = \sqrt{k} a(n/k) \frac{a(\Ynk)}{a(n/k)}
    = \lambda + o_p(1), \qquad n \to \infty.
\end{equation}
Since $a(y) \sim \delta(U(y))$ as $n \to \infty$, this also shows that $\sqrt{k} \delta(X_{n-k:n}) = \lambda + o_p(1)$ as $n \to \infty$.

In the next three paragraphs, we
will analyse the components $\Gamma_{k,n}$, $\EE_{k,n}$ and
$Z_{k,n}$ separately. In the fourth and final paragraph, these
analyses will be combined.

\textit{1. The component $\Gamma_{k,n}$.}
Let the function $a$ be as in \eqref{E:a} and define $\eta(y) = \log \{1 + a(y) \}$. Since $\lim_{y \to \infty} a(y) = 0$, we have $\eta(y) = a(y) \{1 + o(1)\}$ as $y \to \infty$, and hence
\begin{equation}
\label{E:etaYnk}
    \sqrt{k} \eta(\Ynk)
    = \lambda + o_p(1), \qquad n \to \infty.
\end{equation} 
In particular, $\eta$ is eventually nonzero and of constant sign, and $|\eta| \in \RV_\rho$. We have
\begin{align*}
    H_{k,n}
    &= \frac{1}{k} \sum_{i=1}^k \log X_{n-k+i:n} - \log X_{n-k:n} \\
    &= \frac{1}{k} \sum_{i=1}^k \log U(Y_i \Ynk) - \log U(\Ynk) \\
    &= \frac{1}{k} \sum_{i=1}^k \{ \gamma \log Y_i + \eta(Y_i \Ynk) - \eta(\Ynk) \}.
\end{align*}
As a consequence,
\begin{align*}
    \Gamma_{k,n}
    &= \sqrt{k} (H_{k,n} - \gamma) \\
    &= \frac{\gamma}{\sqrt{k}} \sum_{i=1}^k (\log Y_i - 1)
    + \sqrt{k} \eta(\Ynk) \frac{1}{k} \sum_{i=1}^k \left( \frac{\eta(Y_i \Ynk)}{\eta(\Ynk)} - 1 \right).
\end{align*}
By the Uniform Convergence Theorem for $\RV_\rho$, for every $x_0 > 0$,
\[
	\lim_{y \to \infty} \sup_{x \geq x_0} \left| \frac{\eta(xy)}{\eta(y)} - x^\rho \right| = 0.
\]
By the last two displays and in view of \eqref{E:etaYnk},
\begin{equation}
\label{E:epsin}
  \max_{i=1, \ldots, k} \left| \frac{\eta(Y_i \Ynk)}{\eta(\Ynk)} - Y_i^\rho \right| = o_p(1), \qquad n \to \infty.
\end{equation}
By \eqref{E:etaYnk} and since $k^{-1} \sum_{i=1}^k Y_i^\rho = (1 - \rho)^{-1} + o_p(1)$ as $k \to \infty$, we find
\begin{equation}
\label{E:emp:G}
    \Gamma_{k,n}
    = \frac{\gamma}{\sqrt{k}} \sum_{i=1}^k (\log Y_i - 1)
    + \lambda \left( \frac{1}{1 - \rho} - 1 \right) + o_p(1), \qquad n \to \infty.
\end{equation}

\textit{2. The component $\EE_{k,n}$.} Recall the notation $\eta(y) = \log \{1 + a(y)\}$, so that $U(y) = C^\gamma y^\gamma \exp\{ \eta(y) \}$. We have
\begin{align*}
    E_{k,n}(s)
    &= \frac{1}{k} \sum_{i=1}^k \left( \frac{X_{n-k+i:n}}{X_{n-k:n}} \right)^s
    = \frac{1}{k} \sum_{i=1}^k \left( \frac{U(Y_i\Ynk)}{U(\Ynk)} \right)^s \\
    &= \frac{1}{k} \sum_{i=1}^k Y_i^{\gamma s} \exp [ s \{ \eta(Y_i \Ynk) - \eta(\Ynk) \} ].
\end{align*}
Writing $\eps_{i,n} = \eta(Y_i \Ynk) / \eta(\Ynk) - Y_i^\rho$, we find
\[
    E_{k,n}(s) = \frac{1}{k} \sum_{i=1}^k Y_i^{\gamma s} \exp \{ s \eta(\Ynk) (Y_i^\rho - 1 + \eps_{i,n})\}.
\]
Recall the elementary inequality $|e^z - 1 - z| \leq (z^2/2) \max(e^z, 1)$, $z \in \RR$. Since $0 < Y_i^{\gamma s} \leq 1$, $0 < Y_i^\rho \leq 1$ and $\max_{i=1, \ldots, n} |\eps_{i,n}| = o_p(1)$ [see \eqref{E:epsin}], we get by \eqref{E:etaYnk},
\begin{multline*}
    \sup_{s \in [s_0, 0]}
    \left| E_{k,n}(s) - \frac{1}{k} \sum_{i=1}^k Y_i^{\gamma s} \{ 1 + s \eta(\Ynk) (Y_i^\rho - 1) \} \right| \\
    = o_p\{|\eta(\Ynk)|\} = o_p(k^{-1/2}), \qquad n \to \infty.
\end{multline*}
For $\theta_0 < 0$, the class of functions $\{f_\theta : \theta \in [\theta_0, 0]\}$ from $[1, \infty)$ to $(0, 1]$ defined by $f_\theta(y) = y^\theta$, $y \geq 1$, satisfies the Glivenko-Cantelli property
\[
    \sup_{\theta \in [\theta_0, 0]} \left| \frac{1}{k} \sum_{i=1}^k Y_i^\theta - \frac{1}{1-\theta} \right|
    = o_p(1), \qquad k \to \infty;
\]
see for instance Example~19.8 in \citet{vdV98} or just use the monotonicity and continuity of $y^\theta$ in $\theta$. In view of \eqref{E:etaYnk}, we obtain
\begin{multline*}
    \sup_{s \in [s_0, 0]}
    \left| E_{k,n}(s) - \frac{1}{k} \sum_{i=1}^k Y_i^{\gamma s}
                - s \eta(\Ynk) \left(\frac{1}{1 - \gamma s - \rho} - \frac{1}{1 - \gamma s}\right) \right| \\
    = o_p\{|\eta(\Ynk)|\} = o_p(k^{-1/2}), \qquad n \to \infty.
\end{multline*}
Using \eqref{E:etaYnk} again, we find
\begin{align*}
    \EE_{k,n}(s)
    &= \sqrt{k} \left( E_{k,n}(s) - \frac{1}{1 - \gamma s} \right) \\
    &= \frac{1}{\sqrt{k}} \sum_{i=1}^k \left( Y_i^{\gamma s} - \frac{1}{1 - \gamma s} \right)
    + s \lambda \left(\frac{1}{1 - \gamma s - \rho} - \frac{1}{1 - \gamma s}\right)
    + \eps_n(s),
\end{align*}
with
\begin{equation}
\label{E:emp:E}
    \sup_{s \in [s_0, 0]} |\eps_n(s)| = o_p(1), \qquad n \to \infty.
\end{equation}

\textit{3. The component $Z_{k,n}$.}
By \eqref{E:delta} and \eqref{E:a}, we find
\[
    y \bF(U(y)) = 1 + o\{|a(y)|\}, \qquad y \to \infty.
\]
As a consequence,
\begin{align*}
    \bF(X_{n-k:n})
    &= \bF(U(\Ynk)) \\
		&= \Ynk^{-1} [ 1 + o_p\{|a(\Ynk)|\} ] \\
    &= \Ynk^{-1} \{ 1 + o_p(k^{-1/2}) \}, \qquad n \to \infty,
\end{align*}
where we used \eqref{E:aYnk} in the last step. We obtain
\begin{align}
\label{E:emp:Z}
    Z_{k,n}
    &= \sqrt{k} \{ (n/k) \bF(X_{n-k:n}) - 1 \} \nonumber \\
    &= \sqrt{k} \{ (n/k) \Ynk^{-1} - 1 \} + o_p(1), \qquad n \to \infty.
\end{align}

\textit{4. Joint convergence.}
Define
\begin{align*}
    \tilde{\Gamma}_k &= \frac{\gamma}{\sqrt{k}} \sum_{i=1}^k (\log Y_i - 1), \\
    \tilde{\EE}_k(s) &= \frac{1}{\sqrt{k}} \sum_{i=1}^k \left( Y_i^{\gamma s} - \frac{1}{1 - \gamma s} \right), \qquad s \in [s_0, 0].
\end{align*}
For $\theta_0 < 0$, the class of functions $\{f_\theta : \theta \in [\theta_0, 0]\}$ defined by $f_\theta(y) = y^\theta$, $y \geq 1$, is Donsker with respect to the Pareto(1) distribution; this follows from Example~19.7 in \citet{vdV98} upon noting that $|y^{\theta_1} - y^{\theta_2}| \leq |\theta_1 - \theta_2| \log y$ for $\theta_1 \leq 0$, $\theta_2 \leq 0$ and $y \geq 1$. As a consequence, in $\RR \times {\cal C}[s_0, 0]$,
\begin{equation}
\label{E:tilde}
    (\tilde{\Gamma}_k, \tilde{\EE}_k) \dto (\tilde{\Gamma}, \tilde{\EE}), \qquad k \to \infty,
\end{equation}
a centered Gaussian process with covariance function
\begin{align*}
    \var \tilde{\Gamma}
    &= \var (\gamma \log Y_1) = \gamma^2, \\
    \cov \{ \tilde{\EE}(s_1), \tilde{\EE}(s_2) \}
    &= \cov ( Y_1^{\gamma s_1}, Y_1^{\gamma s_2} )
    =  \frac{1}{1 - s_1 \gamma - s_2 \gamma} - \frac{1}{(1 - s_1 \gamma)(1 - s_2 \gamma)}, \\
    \cov \{ \tilde{\Gamma}, \tilde{\EE}(s) \}
    &= \cov ( \gamma \log Y_1, Y_1^{\gamma s} )
    = \frac{s \gamma^2}{(1 - s \gamma)^2}.
\end{align*}
By \eqref{E:emp:G} and \eqref{E:emp:E}, it follows that in $\RR \times {\cal C}[s_0, 0]$,
\[
    (\Gamma_{k,n}, \EE_{k,n}) \dto (\Gamma, \EE), \qquad n \to \infty.
\]
Finally, from \eqref{E:Ynk} and \eqref{E:emp:Z} it follows that $(\Gamma_{k,n}, \EE_{k,n}, Z_{k,n}) \dto (\Gamma, \EE, Z)$ as $n \to \infty$, where $Z$ is standard normally distributed and is independent of $(\Gamma, \EE)$.
\end{proof}

\section{Proof of Theorem~\ref{T:estim}}
\label{A:estim}

The fact that $\sqrt{k} \delta(X_{n-k:n}) = \lambda + o_p(1)$ as $n \to \infty$ has already been shown in the proof of Theorem~\ref{T:simpler}; in particular, see \eqref{E:aYnk}. Recall $\Gamma_{k,n}$ and $\EE_{k,n}(s)$ in equations~\eqref{E:Gammakn} and \eqref{E:EEkn}, respectively, and write $\hat{\tau}_{k,n} = \hat{\rho}_n / H_{k,n}$. We have
\begin{align*}
    \lefteqn{
    \sqrt{k} \left( E_{k,n}(\hat{\tau}_{k,n}) - \frac{1}{1 - \hat{\rho}_n} \right)
    } \\
    &= \sqrt{k} \left( E_{k,n}(\hat{\tau}_{k,n}) - \frac{1}{1 - \gamma \hat{\tau}_n} \right)
    + \sqrt{k} \left( \frac{1}{1 - \gamma \hat{\rho}_n / H_{k,n}} - \frac{1}{1 - \hat{\rho}_n} \right) \\
    &= \EE_{k,n}(\hat{\tau}_{k,n})
    - \frac{1}{H_{k,n}} \frac{\hat{\rho}_n}{(1 - \gamma \hat{\rho}_{k,n} / H_{k,n} )(1 - \hat{\rho}_n)} \Gamma_{k,n}.
\end{align*}
By Theorem~\ref{T:simpler}, $H_{k,n} = \gamma + k^{-1/2} \Gamma_{k,n} = \gamma + o_p(1)$ and thus $\hat{\tau}_{k,n} = \tau + o_p(1)$  as $n \to \infty$. It follows that
\[
    \sqrt{k} \left( E_{k,n}(\hat{\tau}_{k,n}) - \frac{1}{1 - \hat{\rho}_n} \right)
    = \EE_{k,n}(\hat{\tau}_{k,n}) - \frac{\rho}{\gamma(1-\rho)^2} \Gamma_{k,n} + o_p(1), \qquad n \to \infty.
\]
Substituting this into the definition of $\hat{\delta}_{k,n}$ yields
\begin{multline}
\label{E:estim:delta}
    \sqrt{k} \hat{\delta}_{k,n}
    = \gamma (1 - 2\rho) (1 - \rho)^3 \rho^{-4}
    \left( \EE_{k,n}(\hat{\tau}_{k,n}) - \frac{\rho}{\gamma(1-\rho)^2} \Gamma_{k,n} \right) + o_p(1), \\
    n \to \infty,
\end{multline}
as well as
\begin{align}
\label{E:estim:gamma}
    \sqrt{k} (\hat{\gamma}_{k,n} - \gamma)
    &= \sqrt{k} \left( H_{k,n} - \hat{\delta}_{k,n} \frac{\hat{\rho}_n}{1 - \hat{\rho}_n} - \gamma \right) \nonumber \\
    &= \Gamma_{k,n} - \sqrt{k} \hat{\delta}_{k,n} \frac{\rho}{1 - \rho} + o_p(1) \nonumber \\
    &= \frac{(1-\rho)^2}{\rho^2} \left( \Gamma_{k,n} - \gamma \frac{1 - 2\rho}{\rho} \EE_{k,n}(\hat{\tau}_{k,n}) \right) + o_p(1),
    \qquad n \to \infty.
\end{align}
From $\hat{\tau}_{k,n} = \tau + o_p(1)$ and Theorem~\ref{T:simpler}, it follows that in $\RR \times {\cal C}[s_0, 0] \times \RR \times \RR$,
\[
    (\Gamma_{k,n}, \EE_{k,n}, Z_{k,n}, \hat{\tau}_{n,k})
    \dto (\Gamma, \EE, Z, \tau), \qquad n \to \infty.
\]
For $s_0 < \tau$, we have $\Pr(s_0 \leq \hat{\tau}_{n,k} \leq 0) \to 1$ as $n \to \infty$, and thus, by the previous display and the continuous mapping theorem,
\[
    (\Gamma_{k,n}, \EE_{k,n}(\hat{\tau}_{k,n}), Z_{k,n})
    \dto (\Gamma, \EE(\tau), Z), \qquad n \to \infty.
\]
In view of \eqref{E:estim:delta} and \eqref{E:estim:gamma}, as $n \to \infty$,
\begin{multline}
\label{E:estim:10}
    \left(
        \sqrt{k} (\hat{\gamma}_{k,n} - \gamma),
        \sqrt{k} \hat{\delta}_{k,n},
        Z_{k,n}
    \right) \\
    \shoveleft{\qquad \dto
    \Biggl(
        \frac{(1-\rho)^2}{\rho^2} \left( \Gamma - \gamma \frac{1 - 2\rho}{\rho} \EE(\tau) \right),} \\	
        \frac{(1-2\rho)(1-\rho)}{\rho^3} \left( - \Gamma + \gamma \frac{(1-\rho)^2}{\rho} \EE(\tau) \right),
        Z
    \Biggr).
\end{multline}
The vector $(\Gamma, \EE(\tau), Z)$ is trivariate normal, with $Z$ standard normal and independent of $(\Gamma, \EE(\tau))$, with $\Gamma$ as in Theorem~\ref{T:simpler}, and with
\begin{align*}
    \E [ \EE(\tau) ]
    &=  \lambda \frac{\rho^2}{\gamma(1-2\rho)(1-\rho)}, &
    \var \{ \EE(\tau) \}
    &=  \frac{\rho^2}{(1-2\rho)(1-\rho)^2}, \\
    \cov \{ \Gamma, \EE(\tau) \}
    &= \gamma \frac{\rho}{(1-\rho)^2}.
\end{align*}
As a consequence, the distribution of the limit vector in \eqref{E:estim:10} is trivariate normal with mean vector $(0, \lambda, 0)'$ and covariance matrix $\Sigma$ as in \eqref{E:Sigma}.

\section{Proofs for Section~\ref{S:prob}}
\label{A:tail}

\paragraph*{Proof of Theorem~\ref{T:tail}}
Put $y_n = x_n / X_{n-k:n}$, recall $\delta_n = \delta(X_{n-k:n})$, and define
\[
    \tilde{p}_n = \bF(X_{n-k:n}) \bG_{\gamma, \delta_n, \tau}(y_n).
\]
Since $k \to \infty$ and $p_n \to 0$ as $n \to \infty$, it is sufficient to prove \eqref{E:tail} with $\hat{p}_n/p_n - 1$ replaced by $\log \hat{p}_n - \log p_n$. Let us write
\[
    \sqrt{k} (\log \hat{p}_n - \log p_n)
    = \sqrt{k} (\log \hat{p}_n - \log \tilde{p}_n) + \sqrt{k} (\log \tilde{p}_n - \log p_n)
\]
and treat the two terms on the right-hand side separately.

\textit{1. The term $\sqrt{k} (\log \tilde{p}_n - \log p_n)$.} We have
\[
    \log \tilde{p}_n - \log p_n
    = \log \bG_{\gamma, \delta_n, \tau}(y_n)
        - \log \frac{\bF(y_n X_{n-k:n})}{\bF(X_{n-k:n})}.
\]
Since $(n/k_n) \bF(y_n X_{n-k:n}) \to q$ and $(n/k_n) \bF(X_{n-k:n}) = 1 + o_p(1)$ as $n \to \infty$,
\[
    \frac{\bF(y_n X_{n-k:n})}{\bF(X_{n-k:n})}
    = \frac{\bF(x_n)}{\bF(X_{n-k:n})}
    = q + o_p(1), \qquad n \to \infty.
\]
Since moreover $\bF$ is monotone and regularly varying of index $-1/\gamma$, this forces $y_n \to y$ as $n \to \infty$ with $y^{-1/\gamma} = q$, or $y = q^{-\gamma} \in (1, \infty)$. By Proposition~\ref{P:EGPD}, we find
\[
    \frac{\log \tilde{p}_n - \log p_n}{\delta_n} = o_p(1), \qquad n \to \infty.
\]
Finally, from $\sqrt{k} \delta_n = \lambda + o_p(1)$ as $n \to \infty$, we can conclude that
\begin{equation}
\label{E:tail:7}
    \sqrt{k} (\log \tilde{p}_n - \log p_n)
    = \sqrt{k} \delta_n \frac{\log \tilde{p}_n - \log p_n}{\delta_n}
    = o_p(1), \qquad n \to \infty.
\end{equation}

\textit{2. The term $\sqrt{k} (\log \hat{p}_n - \log \tilde{p}_n)$.} We have
\begin{multline}
\label{E:tail:10}
    \log \hat{p}_n - \log \tilde{p}_n
    = \{\log (k/n) - \log \bF(X_{n-k:n})\}  \\
     + \{\log \bG_{\hat{\gamma}_n, \hat{\delta}_n, \hat{\tau}_n}(y_n) - \log \bG_{\gamma, \delta_n, \tau}(y_n)\}.
\end{multline}
The first term on the right-hand side is
\begin{align*}
    \log (k/n) - \log \bF(X_{n-k:n})
    &= - \log \{ n \bF(X_{n-k:n}) / k \} \\
    &= - \log (1 + k^{-1/2} Z_n) \\
    &= - k^{-1/2} Z_n + o_p(k^{-1/2}), \qquad n \to \infty.
\end{align*}
For the second term on the right-hand side in \eqref{E:tail:10}, we proceed as follows. Since $y_n = y + o_p(1)$ as $n \to \infty$ and $y > 1$, it is sufficient to work on the event $y_n > 1$. Then
\begin{align}
\label{E:tail:20}
    \lefteqn{
    \log \bG_{\hat{\gamma}_n, \hat{\delta}_n, \hat{\tau}_n}(y_n) - \log \bG_{\gamma, \delta_n, \tau}(y_n)
    } \nonumber \\
    &= \log [ \{ y_n (1 + \hat{\delta}_n - \hat{\delta}_n y_n^{\hat{\tau}_n} ) \}^{-1/\hat{\gamma}_n} ]
    - \log [ \{ y_n (1 + \delta_n - \delta_n y_n^\tau ) \}^{-1/\gamma} ] \nonumber \\
    &= \left((-\frac{1}{\hat{\gamma}_n}) - (-\frac{1}{\gamma}) \right) \log y_n \nonumber \\
    & \qquad \mbox{} + (-\frac{1}{\hat{\gamma}_n}) \{ \log (1 + \hat{\delta}_n - \hat{\delta}_n y_n^{\hat{\tau}_n})
    - \log (1 + \delta_n - \delta_n y_n^\tau) \} \nonumber \\
    & \qquad \mbox{} + \left((-\frac{1}{\hat{\gamma}_n}) - (-\frac{1}{\gamma}) \right) \log (1 + \delta_n - \delta_n y_n^\tau).
\end{align}
We treat the three terms on the right-hand side of \eqref{E:tail:20} in turn. First,
\begin{align*}
    (-\frac{1}{\hat{\gamma}_n}) - (-\frac{1}{\gamma})
    &= \frac{\hat{\gamma}_n - \gamma}{\hat{\gamma}_n \gamma} \\
    &= k^{-1/2} \gamma^{-2} \Gamma_n + O_p(k^{-1}), \qquad n \to \infty.
\end{align*}
Second, $\delta_n = O_p(k^{-1/2})$ and therefore also $\hat{\delta}_n = O_p(k^{-1/2})$ as $n \to \infty$. Hence the second term on the right-hand side of \eqref{E:tail:20} is
\begin{align*}
    \lefteqn{
    -\hat{\gamma}_n^{-1} \{ \log (1 + \hat{\delta}_n - \hat{\delta}_n y_n^{\hat{\tau}_n})
    - \log (1 + \delta_n - \delta_n y_n^\tau) \}
    } \\
    &= \{- \gamma^{-1} + O_p(k^{-1/2}) \}
        \{ \hat{\delta}_n - \hat{\delta}_n y_n^{\hat{\tau}_n} - \delta_n + \delta_n y_n^\tau + O_p(k^{-1}) \} \\
    &= - k^{-1/2} \gamma^{-1} \Delta_n (1 - y_n^\tau) + o_p(k^{-1/2}), \qquad n \to \infty.
\end{align*}
The third term on the right-hand side of \eqref{E:tail:20} is $O_p(k^{-1/2}) O_p(k^{-1/2}) = O_p(k^{-1})$. All in all, we find
\begin{equation}
\label{E:tail:30}
    \sqrt{k} (\log \hat{p}_n - \log \tilde{p}_n)
    = - Z_n + \gamma^{-2} \Gamma_n \log y - \gamma^{-1} \Delta_n (1 - y^\tau) + o_p(1)
\end{equation}
as $n \to \infty$. Combine \eqref{E:tail:20} and \eqref{E:tail:30} and recall $y = q^{-\gamma}$ and $\rho = \gamma \tau$ to find the result.
\hfill $\Box$

\paragraph*{Proof of Theorem~\ref{T:tail:extreme}}
Recall the Weissman estimator $\hat{p}_n^{\mathrm{W}}$ in
\eqref{E:Weissman} and put $d_n = k / (np_n)$. From Theorem~4.4.7 in \citet{dHF}, it
follows that
\[
    \frac{\sqrt{k}}{\log d_n} \left( \frac{\hat{p}_n^{\mathrm{W}}}{p_n} - 1 \right)
    \dto \gamma^{-1} \Gamma, \qquad n \to \infty.
\]
Moreover, writing $y_n = x_n / X_{n-k:n}$,
\[
    \frac{\hat{p}_n}{\hat{p}_n^{\mathrm{W}}}
    = \{ 1 + \hat{\delta}_n - \hat{\delta}_n y_n^{\hat{\tau}_n} \}^{-1/\hat{\gamma}_n}
    = 1 + O_p(k^{-1/2}), \qquad n \to \infty.
\]
As $\log d_n \to \infty$, we find that $\hat{p}_n$ and $\hat{p}_n^{\mathrm{W}_n}$ have the same asymptotic distribution.
\hfill $\Box$

\section*{Acknowledgments}
We are grateful to two referees for their speedy reports containing thoughtful and constructive remarks.

\bibliographystyle{elsart-harv}
\bibliography{Biblio}

\begin{thebibliography}{33}
\expandafter\ifx\csname natexlab\endcsname\relax\def\natexlab#1{#1}\fi
\expandafter\ifx\csname url\endcsname\relax
  \def\url#1{\texttt{#1}}\fi
\expandafter\ifx\csname urlprefix\endcsname\relax\def\urlprefix{URL }\fi

\bibitem[{Balkema and de~Haan(1974)}]{BalkemadeHaan74}
Balkema, A.~A., de~Haan, L., 1974. Residual life time at great age. The Annals
  of Probability 2, 792--804.

\bibitem[{Beirlant et~al.(2006)Beirlant, Bouquiaux, and Werker}]{BBW06}
Beirlant, J., Bouquiaux, C., Werker, B.~J., 2006. Semiparametric lower bounds
  for tail index estimation. Journal of Statistical Planning and Inference 136,
  705--729.

\bibitem[{Beirlant et~al.(1999)Beirlant, Dierckx, Goegebeur, and
  Matthys}]{BDGM99}
Beirlant, J., Dierckx, G., Goegebeur, Y., Matthys, G., 1999. Tail index
  estimation and an exponential regression model. Extremes 2, 177--200.

\bibitem[{Beirlant et~al.(2002{\natexlab{a}})Beirlant, Dierckx, Guillou, and
  St\u{a}ric\u{a}}]{DBGS02}
Beirlant, J., Dierckx, G., Guillou, A., St\u{a}ric\u{a}, C.,
  2002{\natexlab{a}}. On exponential representations of log-spacings of extreme
  order statistics. Extremes 5, 157--180.

\bibitem[{Beirlant et~al.(2004)Beirlant, Goegebeur, Segers, and Teugels}]{BGST}
Beirlant, J., Goegebeur, Y., Segers, J., Teugels, J., 2004. Statistics of
  Extremes: Theory and Applications. Wiley, Chichester.

\bibitem[{Beirlant et~al.(2002{\natexlab{b}})Beirlant, Joossens, and
  Segers}]{BJS02}
Beirlant, J., Joossens, E., Segers, J., 2002{\natexlab{b}}. Modelling excesses
  over high thresholds by perturbed generalized {Pareto} distributions. Tech.
  Rep. 2002-030, {EURANDOM}, Eindhoven.

\bibitem[{Bingham et~al.(1987)Bingham, Goldie, and Teugels}]{BGT}
Bingham, N.~H., Goldie, C.~M., Teugels, J.~L., 1987. Regular Variation.
  Cambridge University Press, Cambridge.

\bibitem[{Coles(2001)}]{Coles01}
Coles, S.~G., 2001. An Introduction to Statistical Modelling of Extreme Values.
  Springer series in statistics. Springer-Verlag, London.

\bibitem[{Cooray and Ananda(2005)}]{CoorayAnanda05}
Cooray, K., Ananda, M.~A., 2005. Modeling actuarial data with a composite
  lognormal-{Pareto} model. Scandinavian Actuarial Journal 5, 321--334.

\bibitem[{Davison and Smith(1990)}]{DavisonSmith90}
Davison, A.~C., Smith, R.~L., 1990. Models for exceedances over high
  thresholds. Journal of the Royal Statistical Society, Series B 52, 393--442.

\bibitem[{de~Haan and Ferreira(2006)}]{dHF}
de~Haan, L., Ferreira, A., 2006. Extreme Value Theory: An Introduction.
  Springer, New York.

\bibitem[{de~Haan and Stadtm\"uller(1996)}]{dHS96}
de~Haan, L., Stadtm\"uller, U., 1996. Generalized regular variation of second
  order. Journal of the Australian Mathematical Society (Series A) 61,
  381--395.

\bibitem[{Drees(1996)}]{Drees96}
Drees, H., 1996. Refined pickands estimators with bias correction.
  Communications in Statistics -- Theory and Methods 25, 837--851.

\bibitem[{Drees(1998)}]{Drees98}
Drees, H., 1998. A general class of estimators of the extreme value index.
  Journal of Statistical Planning and Inference 66, 95--112.

\bibitem[{Drees et~al.(2004)Drees, Ferreira, and de~Haan}]{DFdH04}
Drees, H., Ferreira, A., de~Haan, L., 2004. On maximum likelihood estimation of
  the extreme value index. The Annals of Applied Probability 14, 1179--1201.

\bibitem[{Feuer\-verger and Hall(1999)}]{FeuervergerHall99}
Feuer\-verger, A., Hall, P., 1999. Estimating a tail exponent by modelling
  departure from a {Pareto} distribution. The Annals of Statistics 27,
  760--781.

\bibitem[{Fraga~Alves et~al.(2003{\natexlab{a}})Fraga~Alves, de~Haan, and
  Lin}]{FAdHL03}
Fraga~Alves, M.~I., de~Haan, L., Lin, T., 2003{\natexlab{a}}. Estimation of the
  parameter controlling the speed of convergence in extreme value theory.
  Mathematical Methods in Statistics 12, 155--176.

\bibitem[{Fraga~Alves et~al.(2003{\natexlab{b}})Fraga~Alves, Gomes, and
  de~Haan}]{FAGdH03}
Fraga~Alves, M.~I., Gomes, M.~I., de~Haan, L., 2003{\natexlab{b}}. A new class
  of semi-parametric estimators of the second order parameter. Portugaliae
  Mathematica 60, 193--214.

\bibitem[{Frigessi et~al.(2002)Frigessi, Haug, and Rue}]{FHR02}
Frigessi, A., Haug, O., Rue, H., 2002. A dynamic mixture model for unsupervised
  tail estimation without threshold selection. Extremes 5, 219--235.

\bibitem[{Geluk and de~Haan(1987)}]{GdH87}
Geluk, J., de~Haan, L., 1987. Regular variation, extensions and tauberian
  theorems. Tech. Rep.~40, CWI tract, Amsterdam.

\bibitem[{Gomes and Martins(2002)}]{GM02}
Gomes, M.~I., Martins, M.~J., 2002. Asymptotically unbiased estimators of the
  tail index based on external estimation of the second order parameter.
  Extremes 5, 5--31.

\bibitem[{Gomes and Martins(2004)}]{GM04}
Gomes, M.~I., Martins, M.~J., 2004. Bias reduction and explicit semi-parametric
  estimation of the tail index. Journal of Statistical Planning and Inference
  124, 361--–378.

\bibitem[{Gomes et~al.(2000)Gomes, Martins, and Neves}]{GMN00}
Gomes, M.~I., Martins, M.~J., Neves, M., 2000. Alternatives to a
  semi-parametric estimator of parameters of rare events -- the {J}ackknife
  methodology. Extremes 3, 207--229.

\bibitem[{Hall(1982)}]{Hall82}
Hall, P., 1982. On some simple estimates of an exponent of regular variation.
  Journal of the Royal Statistical Society, Series B 44, 37--42.

\bibitem[{Hill(1975)}]{Hill75}
Hill, B.~M., 1975. A simple approach to inference about the tail of a
  distribution. The Annals of Statistics 3, 1163--1174.

\bibitem[{Peng(1998)}]{Peng98}
Peng, L., 1998. Asymptotically unbiased estimators for the extreme-value index.
  Statistics \& Probability Letters 38, 107--115.

\bibitem[{Peng and Qi(2004)}]{PengQi04}
Peng, L., Qi, Y., 2004. Estimating the first- and second-order parameters of a
  heavy-tailed distribution. Australian and New-Zealand Journal of Statistics
  46~(2), 305--312.

\bibitem[{Pickands(1975)}]{Pickands75}
Pickands, J., 1975. Statistical inference using extreme order statistics. The
  Annals of Statistics 3, 119--131.

\bibitem[{Reiss(1989)}]{Reiss89}
Reiss, R.-D., 1989. Approximate Distributions of Order Statistics. Springer
  Series in Statistics. Springer, New York.

\bibitem[{Segers(2005)}]{Segers05}
Segers, J., 2005. Generalized {P}ickands estimators for the extreme value
  index. Journal of Statistical Planning and Inference 28, 381--396.

\bibitem[{Smith(1987)}]{Smith87}
Smith, R.~L., 1987. Estimating tails of probability distributions. The Annals
  of Statistics 15, 1174--1207.

\bibitem[{van~der Vaart(1998)}]{vdV98}
van~der Vaart, A., 1998. Asymptotic Statistics. Cambridge University Press,
  Cambridge.

\bibitem[{Weissman(1978)}]{Weissman78}
Weissman, I., 1978. Estimation of parameters and larger quantiles based on the
  $k$ largest observations. Journal of the American Statistical Association 73,
  812--815.

\end{thebibliography}

\end{document}